\documentclass{article}
\usepackage{hyperref}
\hypersetup{hypertexnames = false, bookmarksdepth = 2, bookmarksopen = true, colorlinks, linkcolor = black, citecolor = black, urlcolor = black, pdfstartview={XYZ null null 1}}

\usepackage[fleqn, leqno]{amsmath}
\usepackage{cleveref}
\usepackage{amsthm}
\usepackage{booktabs}
\usepackage{diagbox}
\usepackage{enumitem}
\usepackage{fixltx2e}
\usepackage{mathtools}
\usepackage{parskip}
\usepackage{thmtools}
\usepackage{tikz-cd}
\usepackage[colorinlistoftodos, textsize = footnotesize]{todonotes}
\usepackage{xparse}
\usepackage{xspace}

\usepackage[T1]{fontenc}
\usepackage[charter]{mathdesign}
\usepackage[scaled]{beramono,berasans}
\usepackage{eucal}
\usepackage{microtype}
\declaretheoremstyle[
  spaceabove = 3pt,
  spacebelow = 3pt,
]{lecture}
\theoremstyle{lecture}
\declaretheorem[numberwithin = section]{theorem}
\declaretheorem[sibling = theorem]{corollary}
\declaretheorem[sibling = theorem]{definition}
\declaretheorem[sibling = theorem]{example}

\declaretheorem[sibling = theorem]{lemma}
\declaretheorem[sibling = theorem]{proposition}
\declaretheorem[sibling = theorem]{remark}
\declaretheorem[sibling = theorem]{convention}

\makeatletter
\def\gitfootnote{\gdef\@thefnmark{}\@footnotetext}
\makeatother

\crefname{convention}{convention}{conventions}

\mathchardef\mhyphen="2D
\newcommand\dash{\nobreakdash-\hspace{0pt}}

\newcommand\bounded{\ensuremath{\mathrm{b}}}

\newcommand\compact{\ensuremath{\mathrm{c}}}

\newcommand\LLL{\ensuremath{\mathbf{L}}}
\newcommand\opp{\ensuremath{\mathrm{op}}}
\newcommand\perf{\ensuremath{\mathrm{perf}}}
\newcommand\RRR{\ensuremath{\mathbf{R}}}

\DeclareMathOperator\Br{Br}
\DeclareMathOperator\Catmod{Mod}
\DeclareMathOperator\catmod{mod}
\DeclareMathOperator\CH{CH}
\DeclareMathOperator\Cl{Cl}
\DeclareMathOperator\codim{codim}
\DeclareMathOperator\coh{coh}
\DeclareMathOperator\Cyc{Z}

\DeclareMathOperator\derived{\mathbf{D}}
\DeclareMathOperator\Div{Div}
\DeclareMathOperator\End{End}
\DeclareMathOperator\fl{fl}

\DeclareMathOperator\HH{H}
\DeclareMathOperator\Hom{Hom}
\DeclareMathOperator\sheafHom{\mathcal{H}\mathit{om}}
\DeclareMathOperator\identity{id}
\DeclareMathOperator\im{im}

\DeclareMathOperator\Inj{Inj}
\DeclareMathOperator\Kzero{K_0}

\DeclareMathOperator\Kone{K_1}

\DeclareMathOperator\Mat{Mat}

\DeclareMathOperator\Ob{Ob}

\DeclareMathOperator\Pic{Pic}

\DeclareMathOperator\Qcoh{Qcoh}
\DeclareMathOperator\rad{rad}

\DeclareMathOperator\supp{supp}
\DeclareMathOperator\Supp{Supp}
\DeclareMathOperator\Spc{Spc}
\DeclareMathOperator\relSpec{\mathbf{Spec}}
\DeclareMathOperator\Spec{Spec}

\DeclareMathOperator\ZZ{Z}

\title{Relative tensor triangular Chow groups for coherent algebras}
\author{Pieter Belmans \and Sebastian Klein}

\begin{document}
\maketitle

\begin{abstract}
  We apply the machinery of relative tensor triangular Chow groups to the action of $\derived(\Qcoh(X))$, the derived category of quasi-coherent sheaves on a noetherian scheme $X$, on the derived category of quasi-coherent $\mathcal{A}$\dash modules $\derived(\Qcoh(\mathcal{A}))$, where $\mathcal{A}$ is a (not necessarily commutative) coherent~$\mathcal{O}_X$\dash algebra. When $\mathcal{A}$ is commutative, we recover the tensor triangular Chow groups of $\mathbf{Spec}(\mathcal{A})$. We also obtain concrete descriptions for integral group algebras and hereditary orders over curves, and we investigate the relation of these invariants to the classical ideal class group of an order. An important tool for these computations is a new description of relative tensor triangular Chow groups as the image of a map in the K-theoretic localization sequence associated to a certain Verdier localization.
\end{abstract}


\tableofcontents

\section{Introduction}
In~\cite{1510.00211}, Klein defined and began the study of \emph{relative tensor triangular Chow groups}, a family of~K\dash theoretic invariants attached to a compactly generated triangulated category~$\mathcal{K}$ with an action of a rigidly-compactly generated tensor triangulated category~$\mathcal{T}$ in the sense of~\cite{MR3181496}. While in~\cite{1510.00211}, they were used to improve upon and extend results of~\cite{MR3423452}, the initial observation of the present work is that they allow us to enter the realm of \emph{noncommutative} algebraic geometry: if~$X$ is a noetherian scheme and~$\mathcal{A}$ a (possibly noncommutative) coherent~$\mathcal{O}_X$-algebra, then the derived category~$\mathcal{K}\coloneqq\derived(\Qcoh(\mathcal{A}))$ admits an action by~$\mathcal{T}\coloneqq\derived(\Qcoh(\mathcal{O}_X))$ which is obtained by deriving the tensor product functor
\begin{equation}
  \begin{aligned}
	  \Qcoh(\mathcal{A}) \times \Qcoh(\mathcal{O}_X)
    &\to \Qcoh(\mathcal{A}) \\
	  (M,F)
    &\mapsto M \otimes_{\mathcal{O}_X} F.
  \end{aligned}
\end{equation}
In this situation, the general machinery of~\cite{1510.00211} gives us abelian groups~$\ZZ^{\Delta}_{i}(X,\mathcal{A})$ and~$\CH^{\Delta}_i(X,\mathcal{A})$, the dimension~$i$ tensor triangular cycle and Chow groups of~$\mathcal{K}$ relative to the action of~$\mathcal{T}$. In the test case where~$\mathcal{A}$ is coherent and commutative, and hence~$\mathcal{A}$ corresponds to a scheme~$\relSpec\mathcal{A}$ and a finite morphism~$\relSpec\mathcal{A}\to X$, we show that~$\ZZ^{\Delta}_{i}(X,\mathcal{A})$ and~$\CH^{\Delta}_i(X,\mathcal{A})$ agree with the dimension~$i$ tensor triangular cycle and Chow groups of~$\derived^\bounded(\relSpec\mathcal{A})$ as defined in~\cite{MR3423452}, and hence with the usual dimension~$i$ cycle and Chow groups of~$\ZZ_i(\relSpec\mathcal{A}),\CH_i(\relSpec\mathcal{A})$ when~$\relSpec\mathcal{A}$ is a regular algebraic variety (see \cref{thmcommrecover}). This computation serves as a motivation to study the groups~$\ZZ^{\Delta}_{i}(X,\mathcal{A})$ and~$\CH^{\Delta}_i(X,\mathcal{A})$ for noncommutative coherent~$\mathcal{A}$.

We obtain computations of both invariants when~$\mathcal{A}$ is a sheaf of hereditary orders on a curve in \cref{sectionorders}, and in particular~$\CH^{\Delta}_i(X,\mathcal{A})$ recovers the classical \emph{stable class group} in this case. We also briefly touch upon the subjects of maximal orders on a surface and orders over a singular base, in the context of noncommutative resolutions of singularities. The case of a finite group algebra over~$\Spec\mathbb{Z}$ is discussed as a final example.

Let us highlight that the main ingredient for the calculations carried out in this article is a new exact sequence which is established in \cref{sectionexseq} for a general  rigidly-compactly generated tensor triangulated category~$\mathcal{T}$ acting on a compactly generated triangulated category~$\mathcal{K}$, and for the case~$\mathcal{K}\coloneqq\derived(\Qcoh(\mathcal{A})), \mathcal{T}\coloneqq\derived(\Qcoh(\mathcal{O}_X))$ gives
\begin{equation}
  0 \to \CH^{\Delta}_{p}(X, \mathcal{A}) \to \Kzero\left(\left(\mathcal{K}_{(p+1)}/\mathcal{K}_{(p-1)}\right)^\compact\right) \to \ZZ^{\Delta}_{p+1}(X, \mathcal{A}).
\end{equation}
The middle term of the sequence is the Grothendieck group of the subcategory of compact objects of a subquotient of the filtration of~$\mathcal{K}$ by dimension of support in~$\Spc(\mathcal{T}^\compact)$.

The article is structured as follows: in \cref{sectionttprelims} we recall all relevant notions from tensor triangular geometry and the definition of relative tensor triangular cycle and Chow groups. We then establish the exact sequence mentioned above in \cref{sectionexseq}. In \cref{sectiondercatalg} we prove some auxilliary results concerning the categories~$\derived(\Qcoh(\mathcal{A}))$ and~$\derived^\bounded(\coh(\mathcal{A}))$, most of which should be known to the experts. In \cref{sectionrelgroupsalg} we discuss the action of~$\derived(\Qcoh(\mathcal{O}_X))$ on~$\derived(\Qcoh(\mathcal{A}))$ and contemplate the definition of tensor triangular cycle and Chow groups in this more specific context, including a map~$\CH^{\Delta}_i(X,\mathcal{A}) \to \CH_i(X)$ for regular~$X$, induced by the forgetful functor~$\derived(\Qcoh(\mathcal{A})) \to \derived(\Qcoh(\mathcal{O}_X))$. We then have a look at commutative coherent~$\mathcal{O}_X$-algebras in \cref{sectioncommcohalg} and carry out our computations for orders in \cref{sectionorders}. The results in \cref{sectionorders} motivate the study of relative Chow groups for coherent~$\mathcal{O}_X$\dash algebras in general, by showing that they agree with various invariants in the literature which were defined in an ad hoc way.

\section{Tensor triangular preliminaries}
\label{sectionttprelims}
In this section we recall the categorical notions we need. None of the following material is new, our main sources are~\cite{MR2827786,MR2806103,MR3181496,1510.00211}.
\subsection{Tensor triangular geometry}
\label{subsection:ttgeometry}
Let us quickly recall the basics of Balmer's tensor triangular geometry. See e.g.~\cite{MR2827786} for a reference that covers all the material we need (and much more).
\begin{definition}
	A \emph{tensor triangulated category} is an essentially small triangulated category~$\mathcal{C}$ equipped with a symmetric monoidal structure~$\otimes$ with unit~$\mathbb{I}$ such that the functors~$a \otimes -$ are exact for all objects~$a \in \mathcal{C}$.
	\label{dfnsmallttcat}
\end{definition}

To every tensor triangulated category~$\mathcal{C}$, we associate its \emph{Balmer spectrum}~$\Spc(\mathcal{C})$, a topological space that is constructed in analogy with the prime ideal spectrum of a commutative ring. By construction of~$\Spc(\mathcal{C})$, every object~$a \in \mathcal{C}$ has a closed \emph{support}~$\supp(a) \subset \Spc(\mathcal{C})$, which satisfies the identities
\begin{itemize}
	\item $\supp(0) = \emptyset$ and~$\supp(\mathbb{I}) = \Spc(\mathcal{C})$,
	\item $\supp(\Sigma a) = \supp(a)$,
	\item $\supp(a \oplus b) = \supp(a) \cup \supp(b)$,
	\item $\supp(a \otimes b) = \supp(a) \cap \supp(b)$,
	\item $\supp(b) \subset \supp(a) \cup \supp(c)$ whenever there is a distinguished triangle
    \begin{equation}
	    a \to b \to c \to \Sigma a.
    \end{equation}
\end{itemize}
for all objects~$a,b,c \in \mathcal{C}$. One can show that, in a precise sense, the space~$\Spc(\mathcal{C})$ and the support function~$\supp$ are optimal among all pairs of spaces and support functions satisfying the above criteria.
\begin{example}
	If~$X$ is a quasi-compact, quasi-separated scheme, then~$\mathcal{C} = \derived^\perf(X)$, the derived category of perfect complexes on~$X$, is a tensor triangulated category with tensor product~$\otimes_{\mathcal{O}_X}^\LLL$. We have~$\Spc(\mathcal{C}) \cong X$ and under this identification the support~$\supp(C^{\bullet})$ of some complex~$C^{\bullet}$ is identified with the complement of the set of points~$x \in X$ such that~$C^{\bullet}_x$ is acyclic, or equivalently with the support of the total cohomology sheaf~$\HH^\bullet(C^{\bullet})\coloneqq\bigoplus_{i}\HH^i(C^{\bullet})$.
	\label{exschemereconstruct}
\end{example}
The spectrum~$\Spc(\mathcal{C})$ is always a \emph{spectral} topological space, i.e.~it is homeomorphic to the prime ideal spectrum of some (usually unknown) commutative ring.
Hence, it makes sense to talk about the Krull (co)-dimension of points in~$\Spc(\mathcal{C})$. For a subset~$S \subset \Spc(\mathcal{C})$, we define
\begin{equation}
  \dim(S)\coloneqq\max_{P \in S} \dim(P) \quad \text{and} \quad \codim (S) \coloneqq \min_{P \in S} \codim(P),
\end{equation}
where we set~$\dim(\emptyset) = - \infty, \codim(\emptyset) = \infty$.

\subsection{Supports in large categories}
Let~$\mathcal{T}$ be a triangulated category.
\begin{definition}
	The category~$\mathcal{T}$ is called a \emph{rigidly-compactly generated tensor triangulated category} if
	\begin{enumerate}[label=(\roman*)]
		\item \emph{$\mathcal{T}$ is compactly generated.}
      We implicitly assume here that~$\mathcal{T}$ has set-indexed coproducts. Note that this implies that~$\mathcal{T}$ is not essentially small.

		\item \emph{$\mathcal{T}$ is equipped with a \emph{compatible closed symmetric monoidal structure}
			\begin{equation}
        \otimes: \mathcal{T} \times \mathcal{T} \to \mathcal{T}
      \end{equation}
			with unit object~$\mathbb{I}$}. Here, a symmetric monoidal structure on~$\mathcal{T}$ is \emph{closed} if for all objects~$A \in \mathcal{T}$ the functor~$A \otimes -$ has a right adjoint~$\underline{\hom}(A,-)$. Note that this condition implies that~$\otimes$ preserves coproducts in both variables. A \emph{compatible} closed symmetric monoidal structure on~$\mathcal{T}$ is one such that the functor~$\otimes$ is exact in both variable and such that the two ways of identifying~$\Sigma(x) \otimes \Sigma(y)$ with~$\Sigma^2(x \otimes y)$ are the same up to a sign. Since adjoints of exact functors are exact (see~\cite[lemma~5.3.6]{MR1812507}) we automatically have that the functor~$\underline{\hom}(A,-)$ is exact for all objects~$A \in \mathcal{T}$.

		\item \emph{$\mathbb{I}$ is compact and all compact objects of~$\mathcal{T}$ are rigid.}
      Let~$\mathcal{T}^\compact \subset \mathcal{T}$ denote the full subcategory of compact objects of~$\mathcal{T}$. Then we require that~$\mathbb{I} \in \mathcal{T}^\compact$ and that all objects~$A$ of~$\mathcal{T}^\compact$ are rigid, i.e.\ for every object~$B \in \mathcal{T}$ the natural map
		  \begin{equation}
        \underline{\circ}: \underline{\hom}(A,\mathbb{I}) \otimes B \cong \underline{\hom}(A,\mathbb{I}) \otimes  \underline{\hom}(\mathbb{I},B) \to \underline{\hom}(A,B),
      \end{equation}
		  is an isomorphism.
	\end{enumerate}
\end{definition}
The subcategory~$\mathcal{T}^\compact$ of a rigidly-compactly generated tensor triangulated category~$\mathcal{T}$ is a tensor triangulated category in the sense of \cref{dfnsmallttcat}. Hence, it makes sense to talk about the spectrum~$\Spc(\mathcal{T}^\compact)$.
\begin{convention}
	Throughout this section we assume that~$\mathcal{T}$ is a compactly-rigidly generated tensor triangulated category.\emph{ We also assume that~$\Spc(\mathcal{T}^\compact)$ is a noetherian topological space.}
	\label{convbasiccat}
\end{convention}

\begin{example}
	If~$X$ is a quasi-compact, quasi-separated scheme, then~$\mathcal{T} = \derived_{Qcoh}(X)$, the derived category of complexes of~$\mathcal{O}_X$-modules with quasi-coherent cohomology is a rigidly-compactly generated tensor triangulated category with tensor product~$\otimes_{\mathcal{O}_X}^\LLL$. The rigid-compact objects are the perfect complexes in~$\mathcal{T}$. By \cref{exschemereconstruct},~$\Spc(\mathcal{T}^\compact) = X$ and the condition of \cref{convbasiccat} hence holds whenever the space~$|X|$ is noetherian, e.g.\ when~$X$ is noetherian. If~$X$ is noetherian and separated,~$\derived_{\Qcoh}(X)$ is equivalent to~$\derived(\Qcoh(\mathcal{O}_X))$.
\end{example}

Rigidly-compactly generated tensor triangulated categories come with an associated support theory that extends the notion of support in an essentially small tensor triangulated category. Let us briefly review the theory as introduced in~\cite{MR2806103}. First recall the concepts of Bousfield and smashing subcategories:
\begin{definition}
	A thick triangulated subcategory~$\mathcal{I} \subset \mathcal{T}$ is \emph{Bousfield} if the Verdier quotient functor~$\mathcal{T} \to \mathcal{T}/\mathcal{I}$ exists and has a right adjoint.  A Bousfield subcategory~$\mathcal{I} \subset \mathcal{T}$ is called \emph{smashing} if the right adjoint of the Verdier quotient functor~$\mathcal{T} \to \mathcal{T}/\mathcal{I}$ preserves coproducts.
\end{definition}
If~$\mathcal{I}$ is a Bousfield subcategory, there exists a localization functor~$L_{\mathcal{I}}: \mathcal{T} \to \mathcal{T}$ (given as the composition of the Verdier quotient~$\mathcal{T} \to \mathcal{T}/\mathcal{I}$ and its right adjoint) such that~$\mathcal{I} = \ker(L_{\mathcal{I}})$ and the composition of functors
\begin{equation}
  \mathcal{I}^{\perp} \rightarrow \mathcal{T} \rightarrow \mathcal{T}/\mathcal{I}
\end{equation}
is an exact equivalence, where~$\mathcal{I}^{\perp}$ is the full subcategory consisting of those~$t \in \Ob(\mathcal{T})$ such that~$\Hom(c,t) = 0$ for all~$c \in \Ob(\mathcal{I})$. A quasi-inverse of the equivalence is given by the right adjoint of the Verdier quotient functor~$\mathcal{T} \to \mathcal{T}/\mathcal{I}$. This says that we can actually realize the Verdier quotient~$\mathcal{T}/\mathcal{I}$ inside of~$\mathcal{T}$ and we will freely (and slightly abusively) confuse~$\mathcal{T}/\mathcal{I}$ with~$\mathcal{I}^{\perp}$. Also recall, that for every object~$a \in \mathcal{T}$ there is a distinguished \emph{localization triangle}
\begin{equation}
  \Gamma_{\mathcal{I}}(a) \to a \to L_{\mathcal{I}}(a) \to \Sigma(\Gamma_{\mathcal{I}}(a))
\end{equation}
which is unique among triangles~$x \to a \to y \to \Sigma(x)$ with~$x \in \mathcal{I}$ and~$y \in \mathcal{I}^{\perp}$, up to unique isomorphism of triangles that restrict to the identity on~$a$. This defines a functor~$\Gamma_{\mathcal{I}}(-)$ on~$\mathcal{T}$ with essential image~$\mathcal{I}$. The functor~$\Gamma_{\mathcal{I}}$ is a \emph{colocalization functor}, i.e.~$\Gamma_{\mathcal{I}}^{\opp}$ is a localization functor on~$\mathcal{T}^{\opp}$.

\begin{definition}
	A triangulated subcategory~$\mathcal{I} \subset \mathcal{T}$ is called
	\begin{itemize}
		\item \emph{$\otimes$-ideal} if~$\mathcal{T} \otimes \mathcal{I} \subseteq \mathcal{I}$.
		\item \emph{smashing ideal} if it is a~$\otimes$-ideal, a Bousfield subcategory and~$\mathcal{I}^{\perp} \subset \mathcal{T}$ is also a~$\otimes$-ideal.
	\end{itemize}
\end{definition}

Smashing ideals are well-behaved: as they are Bousfield subcategories there exists a unique triangle
\begin{equation}
  \Gamma_{\mathcal{I}}(\mathbb{I}) \to \mathbb{I} \to L_{\mathcal{I}}(\mathbb{I}) \to \Sigma(\Gamma_{\mathcal{I}}(\mathbb{I})),
\end{equation}
and by tensoring this triangle with~$a \in \mathcal{T}$, we see that we must have~$L_{\mathcal{I}}(a) = L_{\mathcal{I}}(\mathbb{I}) \otimes a$ and~$\Gamma_{\mathcal{I}}(a) = \Gamma_{\mathcal{I}}(\mathbb{I}) \otimes a$.
\begin{remark}
	Smashing ideals are smashing subcategories:~$L_{\mathcal{I}} =  \Gamma_{\mathcal{I}}(\mathbb{I}) \otimes -$ preserves coproducts since it has a right adjoint by definition of a rigidly-compactly generated tensor triangulated category. It follows that the Verdier quotient functor~$\mathcal{T} \to \mathcal{T}/\mathcal{I}$ must preserve coproducts as well.
\end{remark}

An important tool for extending the notion of support from~$\mathcal{T}^\compact$ to~$\mathcal{T}$ is the following theorem:
\begin{theorem}[see {\cite[theorem 3.3.3]{MR1388895}}]
	Let~$\mathcal{S} \subset \mathcal{T}^\compact$ be a thick~$\otimes$-ideal in~$\mathcal{T}^\compact$ (i.e.~$\mathcal{T}^\compact \otimes \mathcal{S} \cong \mathcal{S}$). Let~$\langle \mathcal{S} \rangle$ denote the smallest triangulated subcategory of~$\mathcal{T}$ that is closed under taking arbitrary coproducts (in~$\mathcal{T}$). Then~$\langle \mathcal{S} \rangle$ is a smashing ideal in~$\mathcal{T}$ and~$\langle\mathcal{S}\rangle^\compact \cong \mathcal{S}$.
	\label{thmmiller}
\end{theorem}

\begin{definition}
	Let~$V \subset \Spc(\mathcal{T}^\compact)$ be a specialization-closed subset. We denote by~$\mathcal{T}_{V}$ the smashing ideal~$\langle \mathcal{T}^\compact_V \rangle$, where~$\mathcal{T}^\compact_V \subset \mathcal{T}^\compact$ is the thick~$\otimes$-ideal~$\lbrace a \in \mathcal{T}^\compact: \supp(a) \subset V \rbrace$. We denote the two associated localization and acyclization functors by~$L_V$ and~$\Gamma_V$.
	\label{deftelescopelocacyc}
\end{definition}

Now let~$x \in \Spc(\mathcal{T}^\compact)$ be a point. The sets~$\overline{\lbrace x \rbrace}$ and~$Y_x\coloneqq \lbrace y: x \notin \overline{\lbrace y \rbrace}\rbrace$ are both specialization-closed.
\begin{definition}[see \cite{MR2806103}]
	Let~$x \in \Spc(\mathcal{T}^\compact)$ and let~$\Gamma_x$ denote the functor given as the composition~$L_{Y_x}\Gamma_{\overline{\lbrace x \rbrace}}$. Then, for an object~$a \in \mathcal{T}$, we define its \emph{support} as
	\begin{equation}
    \supp(a) \coloneqq \lbrace x \in  \Spc(\mathcal{T}^\compact) : \Gamma_x(a) \neq 0\rbrace.
  \end{equation}
	\label{dfnbigsupport}
\end{definition}

\begin{example}[see {\cite{MR3177268}}]
	Suppose~$X = \Spec A$ is an affine scheme with~$A$ a noetherian ring. Then~$\derived_{\Qcoh}(\Spec A) \cong \derived(\Catmod(A))$ and
	\begin{equation}
    \Spc(\derived(\Catmod(A))^\compact) = \Spc(\derived^{\mathrm{perf}}(A)) = \Spec A.
  \end{equation}
	Let~$\mathfrak{p} \in \Spec A$ be a prime ideal. Then the functor~$\Gamma_{\mathfrak{p}}$ is given as~$\mathrm{K}_{\infty}(\mathfrak{p}) \otimes A_{\mathfrak{p}} \otimes -$, where~$\mathrm{K}_{\infty}(\mathfrak{p})$ is the \emph{stable Koszul complex} of the prime ideal~$\mathfrak{p}$. In particular, if~$\Supp(C^{\bullet})$ denotes the complement of the set of points where~$C^{\bullet}$ is acyclic, then we see that~$\supp(C^{\bullet}) \subset \Supp(C^{\bullet})$. The set~$\supp(C^{\bullet})$ is sometimes known as the \emph{small support} of~$C^{\bullet}$ and coincides with the set of prime ideals~$\mathfrak{p}$ such that~$k(\mathfrak{p}) \otimes^\LLL C^{\bullet} \neq 0$.
\end{example}

\begin{remark}
	In comparison to the notion of support of an essentially small tensor triangulated category, the support of an object of~$\mathcal{T}$ is still a well-behaved construction. For example, we have~$\supp(\bigoplus_i a_i) = \bigcup_i \supp(a_i)$, but~$\supp(a)$ needs not be closed. If~$a \in \mathcal{T}^\compact$, then~$\supp(a)$ coincides with the notion of support from \cref{subsection:ttgeometry} and hence it will be closed.
\end{remark}

\subsection{Relative supports and tensor triangular Chow groups}
\label{subsection:relative-tensor-triangular-chow-groups}
We shall now adapt to a situation where we consider triangulated categories~$\mathcal{K}$ that don't necessarily have a symmetric monoidal structure themselves, but rather admit an \emph{action} by a tensor triangulated category~$\mathcal{T}$. Let us recall from~\cite{MR3181496} what it means for~$\mathcal{T}$ to have an action~$\ast$ on~$\mathcal{K}$.

We are given a biexact bifunctor
\begin{equation}
  \ast: \mathcal{T} \times \mathcal{K} \to \mathcal{K}
\end{equation}
that commutes with coproducts in both variables, whenever they exist. Furthermore we are given natural isomorphisms
\begin{equation}
	\begin{aligned}
		\alpha_{x,y,a}:& (X \otimes Y) \ast a \overset{\sim}{\longrightarrow} x \ast (y \ast a)\\
		l_a:& \mathbb{I} \ast a \overset{\sim}{\longrightarrow} a
	\end{aligned}
	\label{eqnassocunitiso}
\end{equation}
for all objects~$x,y \in \mathcal{T}, a \in \mathcal{K}$. These natural isomorphisms should satisfy a list of natural coherence relations that we omit here, but rather refer the reader to~\cite{MR3181496}.
\begin{example}
	Any rigidly-compactly generated tensor triangulated category has an action on itself via its monoidal structure.
\end{example}

Let us now assume that we are given a tensor triangulated category~$\mathcal{T}$ with an action~$\ast$ on a triangulated category~$\mathcal{K}$, where~$\mathcal{K}$ is assumed to be compactly generated as well (and so we implicitly mean that it has all coproducts). As in the previous section, we still assume that~$\Spc(\mathcal{T}^\compact)$ is a noetherian topological space. Let us first describe a procedure to construct smashing subcategories of~$\mathcal{K}$.
\begin{lemma}
	Suppose~$V \subset \Spc(\mathcal{T})$ is a specialization-closed subset. Then the full subcategory
	\begin{equation}
    \Gamma_V(\mathbb{I}) \ast \mathcal{K} = \lbrace a \in \mathcal{K}: a \cong \Gamma_V(\mathbb{I}) \ast b~\text{for some}~b \in \mathcal{K}\rbrace
  \end{equation}
	is smashing. The corresponding localization and colocalization functors are given by~$L_V(\mathbb{I}) \ast -$ and~$\Gamma_V(\mathbb{I}) \ast -$, respectively.
	\label{lmarelativesmash}
\end{lemma}
\begin{proof}
  It is shown in~\cite[lemma~4.4]{MR3181496} that the subcategory~$\Gamma_V(\mathbb{I}) \ast \mathcal{K}$ is Bousfield with
	\begin{equation}
    (\Gamma_V(\mathbb{I}) \ast \mathcal{K})^{\perp} = L_V(\mathbb{I}) \ast \mathcal{K} \coloneqq \lbrace a \in \mathcal{K}: a \cong L_V(\mathbb{I}) \ast b~\text{for some}~b \in \mathcal{K}\rbrace.
  \end{equation}
	Both~$\Gamma_V(\mathbb{I}) \ast \mathcal{K}$ and~$L_V(\mathbb{I}) \ast \mathcal{K}$ are~$\mathcal{T}$-submodules, and we have a localization triangle
	\begin{equation}
    \Gamma_V(\mathbb{I}) \to \mathbb{I} \to L_V(\mathbb{I}) \to \Sigma(\Gamma_V(\mathbb{I})).
  \end{equation}
	Applying the functor~$- \ast a$ to this triangle shows that the localization and colocalization functors associated to the Bousfield subcategory are given by~$L_V(\mathbb{I}) \ast -$ and~$\Gamma_V(\mathbb{I}) \ast -$, respectively. Since~$L_V(\mathbb{I}) \ast -$ preserves coproducts by defintion of an action, it follows that~$\Gamma_V(\mathbb{I}) \ast \mathcal{K}~$ is indeed smashing.
\end{proof}

Following~\cite{MR3181496}, we can now assign to any object~$a \in \mathcal{K}$ a support in~$\Spc(\mathcal{T}^\compact)$ as follows:
\begin{definition}
	Let~$x \in \Spc(\mathcal{T}^\compact)$. Then, for an object~$a \in \mathcal{K}$, we define its \emph{support} as
	\begin{equation}
    \supp_{\mathcal{T}}(a) \coloneqq \lbrace x \in \Spc(\mathcal{T}^\compact) : \Gamma_x(\mathbb{I}) \ast a \neq 0\rbrace.
  \end{equation}
\end{definition}
If there is no risk of confusion, we will usually drop the subscript~$\mathcal{T}$ and write~$\supp(a)$ instead. Furthermore, we will abbreviate the expression~$\Gamma_x(\mathbb{I}) \ast a~$ by~$\Gamma_x a$.

Let us state two important properties of the support.
\begin{proposition}[see {\cite[lemma~5.7]{MR3181496}}]
	Let~$V$ be a specialization-closed subset of~$\Spc(\mathcal{T}^\compact)$ and~$a$ an object of~$\mathcal{K}$. Then
	\begin{equation}
    \supp(\Gamma_V(a)) = \supp(a) \cap V
  \end{equation}
	and
	\begin{equation}
    \supp(L_V(a)) = \supp(a) \cap (\Spc(\mathcal{T}^\compact) \setminus V).
  \end{equation}
	\label{propstevesuppprop}
\end{proposition}

\begin{definition}
	For every specialization-closed subset~$V\subset \Spc(\mathcal{T}^\compact)$, the subcategory~$\mathcal{K}_V$ is defined as the essential image of the functor~$\Gamma_V(\mathbb{I}) \ast -$. For every integer~$p$ the subcategory~$\mathcal{K}_{(p)}$ is defined as~$\Gamma_{V_\leq p}(\mathbb{I}) \ast \mathcal{K}$, where~$V_{\leq p} \subset \Spc(\mathcal{T}^\compact)$ is the subset of all points~$x$ such that~$\dim(x) \leq p$.
	\label{dfnsubsetsubcat}
\end{definition}

\begin{remark}
	In~\cite{1510.00211},~$\mathcal{K}_{(p)}$ is defined differently, namely as the full subcategory of~$\mathcal{K}$ on the collection of objects~$\lbrace a \in \mathcal{K}: \dim(\supp(a)) \leq p \rbrace$. This coincides with \cref{dfnsubsetsubcat} whenever~$\supp$ detects vanishing, i.e.~whenever~$\supp(a) = \emptyset \Leftrightarrow a= 0$ holds. Indeed, if~$a \in \Gamma_{V_\leq p}(\mathbb{I}) \ast \mathcal{K}$, then~$a \cong \Gamma_{V \leq p}(\mathbb{I}) \ast b$ for some~$b \in \mathcal{K}$ and it follows from \cref{propstevesuppprop} that~$\supp(a) \subset V_{\leq p}$. Conversely, if
	\begin{equation}
    \dim(\supp(a)) \leq p \Leftrightarrow \supp(a) \subset V_{\leq p},
  \end{equation}
	we have a localization triangle
	\begin{equation}
    \Gamma_{V_{\leq p}}(\mathbb{I}) \ast a \to a \to L_{V_{\leq p}}(\mathbb{I}) \ast a \to \Sigma(\Gamma_{V_{\leq p}}(\mathbb{I})),
  \end{equation}
	and it follows from \cref{propstevesuppprop} that~$\supp(L_{V_{\leq p}}(\mathbb{I})) = \emptyset$ and hence~$L_{V_{\leq p}}(\mathbb{I}) \cong 0$. This implies~$\Gamma_{V_{\leq p}}(\mathbb{I}) \ast a \cong a$ and shows that~$a \in \Gamma_{V_{\leq p}}(\mathbb{I}) \ast \mathcal{K}$. By~\cite[theorem 6.9]{MR3181496},~$\supp$ detects vanishing when the action of~$\mathcal{T}$ on~$\mathcal{K}$ satisfies the \emph{local-to-global principle}, see \cref{remlocglob}.
	\label{remdifferentsubcatdef}
\end{remark}

\begin{proposition}[See {\cite[corollary~4.11]{MR3181496}}]
	Let~$V\subset \Spc(\mathcal{T}^\compact)$ be specialization-closed. The category~$\mathcal{K}_V$ is compactly generated.
	\label{propsubsetsubcatcompgen}
\end{proposition}

We now come to the definition of the central invariant that is studied in this article. For a triangulated category~$\mathcal{C}$, we shall denote by~$\mathcal{C}^{\natural}$ its \emph{idempotent completion}, a triangulated category with a fully faithful inclusion~$\mathcal{C} \to \mathcal{C}^{\natural}$ which is universal for the property that all idempotents in~$\mathcal{C}^{\natural}$ split (see~\cite{MR1813503} for a detailed discussion). Let us first write down a diagram of Grothendieck groups:
\begin{equation}
  \begin{tikzcd}
    \Kzero(\mathcal{K}^\compact_{(p)}) \arrow[r, "q^\natural"] \arrow[d, "i"] & \Kzero((\mathcal{K}^\compact_{(p)}/\mathcal{K}^\compact_{(p-1)})^{\natural}) ~(= \Kzero((\mathcal{K}_{(p)}/\mathcal{K}_{(p-1)})^\compact))\\
  	\Kzero(\mathcal{K}^\compact_{(p+1)})
  \end{tikzcd}
\end{equation}
Here,~$q^{\natural}$ is the map induced by the composition of the Verdier quotient functor~$\mathcal{K}^\compact_{(p)} \to \mathcal{K}^\compact_{(p)}/\mathcal{K}^\compact_{(p-1)}$ and the inclusion into the idempotent completion of the latter category. The morphism~$i$ is induced by the inclusion functor. The identification
\begin{equation}
  (\mathcal{K}^\compact_{(p)}/\mathcal{K}^\compact_{(p-1)})^{\natural} \cong  (\mathcal{K}_{(p)}/\mathcal{K}_{(p-1)})^\compact
\end{equation}
holds by~\cite[theorem~5.6.1]{MR2681709} since~$\mathcal{K}_{(p-1)}$ is compactly generated by \cref{propsubsetsubcatcompgen}.

\begin{definition}[See \cite{1510.00211}]
	The \emph{dimension~$p$ tensor triangular cycle group} of~$\mathcal{K}$ relative to the action~$\ast$ is defined as
	\begin{equation}
    \Cyc^{\Delta}_{p}(\mathcal{T},\mathcal{K}) \coloneqq \Kzero((\mathcal{K}^\compact_{(p)}/\mathcal{K}^\compact_{(p-1)})^{\natural}).
  \end{equation}
	The \emph{dimension~$p$ tensor triangular Chow group} of~$\mathcal{K}$ relative to the action~$\ast$ is defined as
	\begin{equation}
    \CH^{\Delta}_{p}(\mathcal{K}) \coloneqq \Cyc^{\Delta}_{p}(\mathcal{K})/q^{\natural}(\ker(i)).
  \end{equation}
	\label{definition:1510.00211}
\end{definition}

\begin{remark}
	In~\cite{1510.00211}, the definition of relative tensor triangular cycle and Chow groups was given under the assumption that another technical condition, the \emph{local-to-global principle}, is satisified. The principle asserts that for any object~$a \in \mathcal{K}$, the smallest localizing subcategory of~$\mathcal{K}$ that is closed under the action of~$\mathcal{T}$ and contains~$a$ equals the smallest localizing subcategory of~$\mathcal{K}$ that is closed under the action of~$\mathcal{T}$ and contains all the objects~$\Gamma_x a$ for all~$x \in \mathrm{Spc}(\mathcal{T}^c)$ (see \cite[definition~6.1]{MR3181496}). While it is not necessary for the statement of \cref{definition:1510.00211} to make sense, the local-to-global principle makes dealing with these invariants easier (see \cref{remdifferentsubcatdef}), and it is satisfied very often. In particular, it will be satisfied in our main case of interest by~\cite[theorem 6.9]{MR3181496}, when we consider actions of the derived category of quasi-coherent sheaves on a noetherian separated scheme. In order to keep the exposition of the chapter a bit lighter, we will not go into further details concerning this topic.
	\label{remlocglob}
\end{remark}

Let us illustrate our definitions with an example that explains the name ``tensor triangular Chow group''. The following theorem is a slight variation of~\cite[corollary~3.6]{1510.00211}, and is based on Quillen's result describing the Chow groups using the coniveau spectral sequence \cite[proposition~5.14]{MR0338129}.
\begin{theorem}
  Let~$X$ be a separated regular scheme of finite type over a field. Consider the action of~$\derived(\Qcoh(\mathcal{O}_X))$ on itself via~$-\otimes^\LLL-$. Then for all~$p \geq 0$, we have isomorphisms
		\begin{align*}
		  \Cyc^{\Delta}_{p}(\derived(\Qcoh(\mathcal{O}_X)),\derived(\Qcoh(\mathcal{O}_X))) &\cong \Cyc_{p}(X) \\
		  \CH^{\Delta}_{p}(\derived(\Qcoh(\mathcal{O}_X)),\derived(\Qcoh(\mathcal{O}_X))) &\cong \CH_{p}(X),
		\end{align*}
		where~$\Cyc_{p}(X)$ and~$\CH_{p}(X)$ denote the dimension~$p$ cycle and Chow groups of~$X$.
		\label{thmchowrecover}
\end{theorem}
\begin{proof}
  This is~\cite[corollary~3.6]{1510.00211}, with codimension replaced by dimension. The former statement is proved by showing that the groups~$\Cyc_{\Delta}^{p}(\derived(\Qcoh(\mathcal{O}_X)),\derived(\Qcoh(\mathcal{O}_X)))$ and~$\CH_{\Delta}^{p}(\derived(\Qcoh(\mathcal{O}_X)),\derived(\Qcoh(\mathcal{O}_X)))$ which are defined analogously via a filtration by \emph{co}dimension of support, are isomorphic to certain terms on the~$E^1$ and~$E^2$ page of Quillen's coniveau spectral sequence associated to~$X$, which happen to be isomorphic to~$\Cyc^{p}(X)$ and~$\CH^{p}(X)$, respectively.

  In order to prove the ``dimension'' version of the statement, we see that the same argument shows that~$\Cyc^{\Delta}_{p}(\derived(\Qcoh(\mathcal{O}_X)),\derived(\Qcoh(\mathcal{O}_X)))$ and~$\CH^{\Delta}_{p}(\derived(\Qcoh(\mathcal{O}_X)),\derived(\Qcoh(\mathcal{O}_X)))$ are isomorphic to the terms~$E^1_{p,-p}$ and~$E^2_{p,-p}$ of the \emph{niveau} spectral sequence of~$X$, which happen to be isomorphic to~$\Cyc_{p}(X)$ and~$\CH_{p}(X)$ (see e.g.~\cite{MR2648734}  for the identification of~$E^1_{p,-p}$ and~$E^2_{p,-p}$ with~$\Cyc_{p}(X)$ and~$\CH_{p}(X)$).
\end{proof}

\begin{remark}[See {\cite[\S4]{1510.00211}}]
	We can actually do better and also recover~$\CH_{p}(X)$ for singular schemes. In order to do so, one lets~$\derived(\Qcoh(\mathcal{O}_X))$ act on~$\mathrm{K}(\Inj X)$, the homotopy category of quasi-coherent injective sheaves on~$X$, instead of considering the action of~$\derived(\Qcoh(\mathcal{O}_X))$ on itself. Later on, we shall be interested in the action of~$\derived(\Qcoh(\mathcal{O}_X))$ on the derived category of a coherent~$\mathcal{O}_X$-algebra.
	\label{remcoderived}
\end{remark}

\section{An exact sequence}
\label{sectionexseq}
In this section we derive an exact sequence that will give us a new description of~$\CH^{\Delta}_{p}(\mathcal{T}, \mathcal{K})$ as an image of a map in a K-theoretic localization sequence. It will be especially useful for computing~$\CH^{\Delta}_{0}(\mathcal{T}, \mathcal{K})$ when~$\dim(\Spc(\mathcal{T}^\compact)) =1$. Let~$\mathcal{T}$ be a rigidly-compactly generated triangulated category that has an action~$\ast$ on a compactly generated triangulated category~$\mathcal{K}$ and assume that~$\Spc(\mathcal{T}^\compact)$ is a noetherian topological space. Then we know that~$\mathcal{K}_{(p)}$ is a compactly generated subcategory of~$\mathcal{K}$ for all~$p \geq 0$ and we have an exact sequence of triangulated categories
\begin{equation}
  \mathcal{K}_{(p)}/\mathcal{K}_{(p-1)} \to \mathcal{K}_{(p+1)}/\mathcal{K}_{(p-1)} \to \mathcal{K}_{(p+1)}/\mathcal{K}_{(p)}.
  \label{eqnexseqdimdiff2}
\end{equation}
Since the inclusion~$\mathcal{K}_{(p)} \to \mathcal{K}_{(p+1)}$ admits a coproduct-preserving right adjoint~$\Gamma_{V_{\leq p}}(\mathbb{I}) \ast -$, the same is true for both functors in the sequence (\ref{eqnexseqdimdiff2}). Hence it restricts to a sequence of compact objects
\begin{equation}
  \left(\mathcal{K}_{(p)}/\mathcal{K}_{(p-1)}\right)^\compact \to \left(\mathcal{K}_{(p+1)}/\mathcal{K}_{(p-1)}\right)^\compact \to \left(\mathcal{K}_{(p+1)}/\mathcal{K}_{(p)}\right)^\compact
\end{equation}
which is exact up to factors. Applying~$\Kzero$ to this diagram yields a sequence of abelian groups
\begin{equation}
  \ZZ^{\Delta}_{p}(\mathcal{T}, \mathcal{K}) \xrightarrow{\iota} \Kzero\left(\left(\mathcal{K}_{(p+1)}/\mathcal{K}_{(p-1)}\right)^\compact\right) \xrightarrow{\pi} \ZZ^{\Delta}_{p+1}(\mathcal{T}, \mathcal{K})
\end{equation}
which is exact at the middle term.

\begin{lemma}
	The map~$\pi$ is surjective if and only if~$\mathcal{K}_{(p+1)}^\compact/\mathcal{K}_{(p)}^\compact$ is idempotent complete.
	\label{lmaidcompsurj}
\end{lemma}
\begin{proof}
	We have~$\left(\mathcal{K}_{(p+1)}/\mathcal{K}_{(p)}\right)^\compact \cong \left(\mathcal{K}_{(p+1)}^\compact/\mathcal{K}_{(p)}^\compact\right)^{\natural}$ and hence~$\mathcal{K}_{(p+1)}^\compact/\mathcal{K}_{(p)}^\compact$ is a dense triangulated subcategory of~$\left(\mathcal{K}_{(p+1)}/\mathcal{K}_{(p)}\right)^\compact$. Thomason's classification of these subcategories (see~\cite{MR1436741}) then shows that~$\im(\pi)$ is maximal if and only if the inclusion ~$\mathcal{K}_{(p+1)}^\compact/\mathcal{K}_{(p)}^\compact \hookrightarrow \left(\mathcal{K}_{(p+1)}/\mathcal{K}_{(p)}\right)^\compact$ is essentially surjective which happens if and only if the former category is idempotent complete.
\end{proof}

We shall now be concerned with the kernel of~$\iota$. Our goal is to prove the following statement:
\begin{proposition}
	\label{propchowexseq}
	In the notation of \cref{definition:1510.00211}, we have~$\ker(\iota) = q^{\natural}(\ker(i))$. Hence, we obtain an exact sequence
	\begin{equation}
    0 \to \CH^{\Delta}_{p}(\mathcal{T}, \mathcal{K}) \xrightarrow{\overline{\iota}} \Kzero\left(\left(\mathcal{K}_{(p+1)}/\mathcal{K}_{(p-1)}\right)^\compact\right) \xrightarrow{\pi} \ZZ^{\Delta}_{p+1}(\mathcal{T}, \mathcal{K})
  \end{equation}
	which is exact on the right if and only if~$\mathcal{K}_{(p+1)}^\compact/\mathcal{K}_{(p)}^\compact$ is idempotent complete.
\end{proposition}

\begin{lemma}
  Let~$\mathcal{K}$ be a triangulated category and~$\mathcal{L} \subset \mathcal{K}$ a thick\footnote{It was pointed out to us by J\o rgen Rennemo that without this condition the statement of this lemma, and the next, is false. The proof of this lemma implicitly used this condition and is not changed from the published version, aside from fixing typos. The condition is satisfied for all situations considered in this paper.} triangulated subcategory. Consider the full triangulated subcategories~$\mathcal{L}^{\natural}, \mathcal{K} \subset \mathcal{K}^{\natural}$. Then~$\mathcal{L}^{\natural} \cap \mathcal{K} \cong \mathcal{L}$ as full subcategories of~$\mathcal{K}^{\natural}$.
	\label{lmacapsubcats}
\end{lemma}
\begin{proof}
	It is clear that an object~$A \in \mathcal{L}$ is both contained in~$\mathcal{L}^{\natural}$ and~$\mathcal{K}$. For the converse inclusion, suppose that~$A$ is in~$\mathcal{L}^{\natural} \cap \mathcal{K}$. Any object~$A \in \mathcal{L}^{\natural}$ can be written as a pair~$(A',e)$, where~$A'$ is an object of~$\mathcal{L}$ and~$e$ is an idempotent endomorphism~$A' \to A'$ in~$\mathcal{L}$. Similarly, the objects~$B$ of~$\mathcal{K}$ in~$\mathcal{K}^{\natural}$ are identified with exactly the pairs~$(B, \identity_B)$. It follows that~$A$ can be written in the form~$(A',\identity_{A'})$ with~$A' \in \mathcal{L}$. Hence,~$A$ is in the image of the inclusion functor~$\mathcal{L} \to \mathcal{K}^{\natural}$.
\end{proof}

\begin{lemma}
	In the situation of \cref{lmacapsubcats}, assume that~$\mathcal{L}, \mathcal{K}$ are essentially small and consider the diagram of Grothendieck groups
	\begin{equation}
	  \begin{tikzcd}
      \Kzero(\mathcal{L}) \arrow{r}{\alpha} \arrow{d}{\rho} & \Kzero(\mathcal{K}) \arrow{d}{\sigma} \\
      \Kzero(\mathcal{L}^{\natural}) \arrow{r}{\beta} &  \Kzero(\mathcal{K}^{\natural})
	  \end{tikzcd}
	\end{equation}
	induced by the inclusion functors. Then~$\ker(\beta) = \rho(\ker(\alpha))$.
	\label{lmaidcompkers}
\end{lemma}
\begin{proof}
	By the commutativity of the diagram, it is clear that~$\ker(\beta) \supseteq \rho(\ker(\alpha))$, so let us prove the converse inclusion. Consider an element~$[a] \in \ker(\beta)$, i.e.~$[a] = 0$ in~$\Kzero(\mathcal{K}^{\natural})$. By Thomason's classification of dense triangulated subcategories (see~\cite{MR1436741}) applied to~$\mathcal{K} \subset \mathcal{K}^{\natural}$, we have
	\begin{equation}
    \mathcal{K} = \lbrace x \in \mathcal{K}^{\natural}: [x] \in \im(\sigma) \rbrace.
  \end{equation}
	Since~$0 \in \im(\sigma)$, we must have~$a \in \mathcal{K} \subset \mathcal{K}^{\natural}$, and by \cref{lmacapsubcats} it follows that~$a \in \mathcal{L}$. Thus,~$[a] \in \im(\rho)$ and since~$\sigma$ is injective (see~\cite[corollary 2.3]{MR1436741}), it follows that~$[a] \in \ker(\alpha)$.
\end{proof}

\begin{proof}[Proof of \cref{propchowexseq}]
	Consider the commutative diagram
	\begin{equation}
    \begin{tikzcd}
      \Kzero(\mathcal{K}_{(p)}^\compact) \arrow[r, "i"] \arrow[d, "q"] & \Kzero(\mathcal{K}_{(p+1)}^\compact) \arrow[d, "h"] \\
      \Kzero\left(\mathcal{K}_{(p)}^\compact/\mathcal{K}_{(p-1)}^\compact\right) \arrow[r, "k"] \arrow[d, "j"] & \Kzero\left(\mathcal{K}_{(p+1)}^\compact/\mathcal{K}_{(p-1)}^\compact\right) \arrow[d, "l"] \\
		  {\underbrace{\Kzero\left(\left(\mathcal{K}_{(p)}^\compact/\mathcal{K}_{(p-1)}^\compact\right)^{\natural}\right)}_{ = \ZZ^{\Delta}_p(\mathcal{T}, \mathcal{K})}} \arrow[r, "\iota"] & {\underbrace{\Kzero\left(\left(\mathcal{K}_{(p+1)}^\compact/\mathcal{K}_{(p-1)}^\compact\right)^{\natural}\right)}_{ = \Kzero\left(\left(\mathcal{K}_{(p+1)}/\mathcal{K}_{(p-1)}\right)^\compact\right)}}
    \end{tikzcd}
	  \label{eqncommdiagsubcats}
	\end{equation}
	where all maps are induced by inclusions of subcategories or Verdier quotient functors and in particular, we have~$q^{\natural} = j \circ q$. Since~$\ker(h) = i(\ker(q))$, we obtain that~$\ker(k) = q(\ker(i))$. Therefore, it suffices to show that~$\ker(\iota) = j(\ker(k))$, which follows from \cref{lmaidcompkers}. The last statement of the proposition is \cref{lmaidcompsurj}.
\end{proof}

\begin{remark}
	When~$\dim(\Spc(\mathcal{T}^\compact)) =1$, \cref{propchowexseq} exhibits~$\CH^{\Delta}_0(\mathcal{T},\mathcal{K})$ as a subgroup of~$\Kzero(\mathcal{K}^\compact)$. If~$X$ is a regular algebraic curve,~$\mathcal{T} = \mathcal{K} =\derived(\Qcoh(\mathcal{O}_X))$, then we recover the well-known isomorphism
	\begin{equation}
    \Kzero(X) \cong \CH_0(X) \oplus \ZZ_1(X)
  \end{equation}
  using \cref{thmchowrecover}: the map~$\pi$ is surjective by \cref{lmaidcompsurj}, since
	\begin{align*}
	\derived^\perf(\coh(X))_{(1)}/\derived^\perf(\coh(X))_{(0)} &\cong \derived^\bounded(\coh(X))/\derived^\bounded(\coh(X))_{(0)} \\
	&\cong \derived^\bounded(\coh(X)/\coh(X)_{\leq 0}),
	\end{align*}
  (see~\cite[\S3.2]{MR3423452}, compare \cref{corverdiervsserre}) and the latter category is idempotent complete since it is the bounded derived category of an abelian category (see~\cite{MR1813503}). Furthermore,~$\ZZ_1(X)$ is free abelian and hence the exact sequence splits. Again, as in \cref{remcoderived}, we can drop the regularity assumption and consider the action of~$\derived(\Qcoh(\mathcal{O}_X))$ on~$\mathrm{K}(\Inj X)$ instead. Then we obtain
	\begin{equation}
    \mathrm{G}_0(X) \cong \CH_0(X) \oplus \ZZ_1(X).
  \end{equation}
\end{remark}

\section{Derived categories of quasi-coherent \texorpdfstring{$\mathcal{O}_X$}{OX}-algebras}
\label{sectiondercatalg}

In this section, we first recall some well-known facts about the category of quasi-coherent right~$\mathcal{A}$-modules~$\Qcoh(\mathcal{A})$, and its derived category~$\derived(\Qcoh(\mathcal{A}))$. We show how to realize the functor~$\derived(\Qcoh(\mathcal{A})) \to \derived(\Catmod(\mathcal{A}_x))$ that takes stalks at~$x \in X$ as a localization of~$\derived(\Qcoh(\mathcal{A}))$ and prove a technical result about the filtration of~$\derived^\bounded(\coh(\mathcal{A}))$ by dimension of support. At this point we will not need to assume that~$\mathcal{A}$ is coherent, quasi-coherence is enough. Starting from \cref{sectionrelgroupsalg} we will impose the coherence condition to make the action well-behaved on the level of compact objects, and to make the different notions of support agree.

\subsection{Basics of quasi-coherent modules over quasi-coherent \texorpdfstring{$\mathcal{O}_X$}{O\textunderscore X}-algebras}
Let~$X$ be a scheme. In this section we recall some basic facts about modules over an~$\mathcal{O}_X$-algebra~$\mathcal{A}$. The material we present here should be well-known (or at least hardly surprising) to most experts.

An~$\mathcal{O}_X$-algebra~$\mathcal{A}$ is a sheaf of~$\mathcal{O_X}$-modules~$\mathcal{A}$ together with a multiplication map~$\mathcal{A} \times \mathcal{A} \to \mathcal{A}$ that is associative and has unit, and is~$\mathcal{O}_X$-bilinear\footnote{This last condition implies that~$\mathcal{O}_X$ acts centrally on~$\mathcal{A}$.}. An~$\mathcal{O}_X$-algebra~$\mathcal{A}$ is \emph{quasi-coherent}, if it is so as an~$\mathcal{O}_X$-module. The pair~$(X,\mathcal{A})$ is a ringed space, and hence it makes sense to talk about quasi-coherent right~$\mathcal{A}$-modules.  It is not hard to show that if~$\mathcal{A}$ is a quasi-coherent~$\mathcal{O}_X$-algebra, then a right~$\mathcal{A}$-module is quasi-coherent if and only if it is quasi-coherent as an~$\mathcal{O}_X$-module. Furthermore, quasi-coherent right~$\mathcal{A}$-modules over a quasi-coherent~$\mathcal{O}_X$-algebra~$\mathcal{A}$ have a local description analogous to quasi-coherent~$\mathcal{O}_X$-modules.
\begin{proposition}[see {\cite[proposition~1.15]{MR2244264}}]
	Let ~$\mathcal{A}$ be a quasi-coherent~$\mathcal{O}_X$-algebra,~$U \subset X$ an affine open and~$A \coloneqq \Gamma(U,\mathcal{A})$. Then the functor~$\Gamma(U,-)$ induces an equivalence of categories
	\begin{equation}
    \lbrace \text{quasi-coherent right~$\mathcal{A}|_{U}$-modules} \rbrace \xrightarrow{\sim} \lbrace \text{right~$A$-modules} \rbrace.
  \end{equation}
\end{proposition}

Since the notion of coherence is general as well, it applies to right~$\mathcal{A}$-modules. We shall primarily be interested in the case where~$X$ is noetherian and~$\mathcal{A}$ is a \emph{coherent~$\mathcal{O}_X$-algebra}, i.e.~one that is coherent as an~$\mathcal{O}_X$-module.
\begin{lemma}
	Suppose~$X$ is noetherian and~$\mathcal{A}$ is a coherent~$\mathcal{O}_X$-algebra. Then a right~$\mathcal{A}$-module~$M$ is coherent if and only if it is coherent as an~$\mathcal{O}_X$-module.
\end{lemma}
\begin{proof}[Sketch of the proof]
	Let us first notice that under the given conditions,~$\mathcal{A}$ is a sheaf of right-noetherian rings. A right~$\mathcal{A}$-module is hence coherent if and only if it is locally of finite type. Therefore, it suffices to show that a right~$\mathcal{A}$-module is locally of finite type over~$\mathcal{A}$ if and only if it is so over~$\mathcal{O}_X$, which is straightforward.
\end{proof}

\begin{proposition}
 The category~$\Qcoh(\mathcal{A})$ is Grothendieck abelian.
 \label{corqcohAGrothendieck}
\end{proposition}
\begin{proof}
  The category~$\Qcoh(\mathcal{A})$ is exactly the category of modules over the right-exact monad corresponding to the adjunction~$\mathcal{A} \otimes_{\mathcal{O}_X} - \dashv U$. Then~\cite[lemma~A.3]{MR3161097} applies and shows that~$\Qcoh(\mathcal{A})$ is Grothendieck abelian, since~$\Qcoh(\mathcal{O}_X)$ is so.
\end{proof}

The following notion is central for our further considerations:
\begin{definition}
	Let~$M \in \Qcoh(\mathcal{O}_X)$. The \emph{support~$\Supp(M)$ of~$M$} is the set of points~$P \in X$ such that~$M_P \neq 0$. If~$N \in \Qcoh(\mathcal{A})$, then~$\Supp(N) \coloneqq \Supp(U(N)) \subset X$.
\end{definition}

\subsection{The derived category of a quasi-coherent \texorpdfstring{$\mathcal{O}_X$}{O\textunderscore X}-algebra}
\label{subsection:derived-quasi-coherent-algebras}
In the following, \emph{we shall always assume that~$X$ is a noetherian separated scheme} and that~$\mathcal{A}$ is a quasi-coherent~$\mathcal{O}_X$-algebra. These are not the strongest possible assumptions for (most of) the results in this section, but in \cref{sectionrelgroupsalg} we will need these (and stronger) conditions to develop the machinery of relative tensor triangular Chow groups.
\subsubsection{Basic properties}
In this section we study the category~$\derived(\Qcoh(\mathcal{A}))$, the derived category of quasi-coherent right-$\mathcal{A}$-modules. Let us first note that~$\derived(\Qcoh(\mathcal{A}))$ exists, since~$\Qcoh(\mathcal{A})$ is Grothendieck abelian by \cref{corqcohAGrothendieck}. Furthermore, since the forgetful functor~$U$ is exact, it directly descends to give a functor~$U : \derived(\Qcoh(\mathcal{A}))\to \derived(\Qcoh(\mathcal{O}_X))$. Its right adjoint~$\mathcal{A} \otimes_{\mathcal{O}_X} -$ induces a left-derived functor
\begin{equation}
  \mathcal{A} \otimes^\LLL_{\mathcal{O}_X} - : \derived(\Qcoh(\mathcal{O}_X)) \to \derived(\Qcoh(\mathcal{A}))
\end{equation}
which is computed by first taking K-flat resolutions in~$\derived(\Qcoh(\mathcal{O}_X))$ and then applying~$\mathcal{A} \otimes_{\mathcal{O}_X} -$.
\begin{proposition}
	There is an adjunction~$(\mathcal{A} \otimes^\LLL_{\mathcal{O}_X} -) \dashv U$.
\end{proposition}
\begin{proof}
  This follows since the derived functors of an adjoint pair, if they exist, are again adjoint \cite[tag~09T5]{stacks-project}.
\end{proof}

\begin{theorem}
	The category~$\derived(\Qcoh(\mathcal{A}))$ is compactly generated, and a complex in~$\derived(\Qcoh(\mathcal{A}))$ is compact if and only if it is perfect, i.e.~it is locally quasi-isomorphic to a bounded complex of projective modules of finite rank.
	\label{thmcompgen}
\end{theorem}
\begin{proof}
  This can be shown using Rouquier's cocoverings. See~\cite[theorem~3.14]{MR3695056}.
\end{proof}

\begin{convention}
In the following, we shall denote the full subcategory of perfect complexes over~$\mathcal{A}$ by~$\derived^\perf(\mathcal{A}) \subset \derived(\Qcoh(\mathcal{A}))$. Whenever~$S \subset |X|$ is a subset, we shall denote by~$\derived_S(\mathcal{A})$ ($\derived^\bounded_S(\coh(\mathcal{A})), \derived^{\mathrm{perf}}_S(\mathcal{A})$) the corresponding full subcategories consisting of complexes~$C^{\bullet}$ with~$\Supp(\HH^\bullet(C^{\bullet})) \subset S$. If~$S=V_{\leq p}$, the subset of all points of dimension~$\leq p$, we shall replace the subscript ``$V_{\leq p}$'' by~``$\leq p$''.
\label{convsubcatfilt}
\end{convention}

\subsubsection{Taking stalks}
Let us consider a point~$x \in X$ and the inclusion~$\Spec\mathcal{O}_{X,x}\to X$. If we equip~$\Spec \mathcal{O}_{X,x}$ with the sheaf of rings~$\mathcal{A}_x$, we obtain a morphism of ringed spaces
\begin{equation}
  i_x\colon(\Spec \mathcal{O}_{X,x}, \mathcal{A}_x) \to (X,\mathcal{A})
\end{equation}
and the general theory of ringed spaces gives us a pair of adjoint functors
\begin{equation}
	\begin{tikzcd}
		\Catmod(\mathcal{A}_x) \arrow[bend right,swap]{d}{(i_x)_*}\\
		\Qcoh(\mathcal{A}) \arrow[bend right,swap]{u}{(i_x)^*}
	\end{tikzcd}
\end{equation}
which fits into a commutative diagram
\begin{equation}
	\begin{tikzcd}
		\Catmod(\mathcal{O}_{X,x}) \arrow[bend right,swap]{d}{(i_x)_*} &  \Catmod(\mathcal{A}_x) \arrow[bend right, swap]{d}{(i_x)_*} \arrow[swap,]{l}{U}\\
		\Qcoh(\mathcal{O}_X) \arrow[bend right,swap]{u}{(i_x)^*} &\Qcoh(\mathcal{A}) \arrow[bend right, swap]{u}{(i_x)^*}  \arrow{l}{U}
	\end{tikzcd}
\end{equation}
and satisfies~$(i_x)^* \circ (i_x)_* = \identity$. The map~$\Spec\mathcal{O}_{X,x}\to X$ is quasi-separated and quasi-compact (recall that we assumed that~$X$ noetherian). Therefore the functor~$(i_x)_*$ indeed produces quasi-coherent~$\mathcal{O}_X$-modules, and hence also quasi-coherent~$\mathcal{A}$-modules, since quasi-coherence can be checked after applying~$U$.

Since~$X$ was separated, the map~$i_x$ is affine and thus the functor~$(i_x)_*$ is exact on the level of~$\mathcal{O}_{X,x}$-modules. Since~$U$ preserves and reflects exactness, it follows that~$(i_x)_*$ is exact on the level of~$\mathcal{A}_x$-modules as well. Furthermore, the map~$\Spec \mathcal{O}_{X,x}\to X$ is flat and hence~$(i_x)^*$ is exact on both levels as well.

Since the derived functors of an adjoint pair are again adjoint \cite[tag~09T5]{stacks-project}, we obtain an adjunction
\begin{equation}
	\begin{tikzcd}
		\derived(\Catmod(\mathcal{A}_x)) \arrow[bend right, swap]{d}{(i_x)_*}\\
		\derived(\Qcoh(\mathcal{A})) \arrow[bend right, swap]{u}{(i_x)^*}
	\end{tikzcd}
\end{equation}
which still satisfies~$(i_x)^* \circ (i_x)_* = \identity$ since there was no need to derive any of the two functors.

\begin{proposition}
	Let~$X$ be a noetherian separated scheme and~$\mathcal{A}$ a quasi-coherent~$\mathcal{O}_X$-algebra. Let~$x \in X$ and~$\derived_{Y_x}(\mathcal{A}) \subset \derived(\Qcoh(\mathcal{A}))$ be the full subcategory of complexes~$C^{\bullet}$ such that~$\Supp(\HH^\bullet(C^{\bullet})) \subset Y_x = \lbrace y \in X | x \notin \overline{\lbrace y \rbrace} \rbrace$. Then~$\derived_{Y_x}(\mathcal{A}) \cong \ker(i_x)^*$ and the functor~$(i_x)^*$ induces an exact equivalence
	\begin{equation}
    \derived(\Qcoh(\mathcal{A}))/\derived_{Y_x}(\mathcal{A}) \xrightarrow{\sim} \derived(\Catmod(\mathcal{A}_x)).
  \end{equation}
	\label{propstalkquot}
\end{proposition}
\begin{proof}
	The first part follows from the identity~$\HH^\bullet((i_x)^*C^{\bullet}) = (i_x)^*(\HH^\bullet(C^{\bullet}))$.

  Since~$(i_x)^* \circ (i_x)_* = \identity$, we must have that~$(i_x)_*$ is fully faithful. It is well-known (see e.g.~\cite[lemma~3.4]{MR2681712}) that we therefore get an exact sequence of triangulated categories
	\begin{equation}
    \derived(\Qcoh(\mathcal{A}))/\ker(i_x)^* \xrightarrow{\sim} \derived(\Catmod(\mathcal{A}_x)),
  \end{equation}
	which finishes the proof by the first part of the proposition.
\end{proof}

\subsubsection{Filtrations of the bounded derived category of coherent sheaves}
Let us now assume that~$X$ is a noetherian scheme and that~$\mathcal{A}$ is a coherent~$\mathcal{O}_X$-algebra. We record the following, essentially trivial lemma for later use.
\begin{lemma}
  \label{lmaidealannihilate}
	Let~$\mathcal{J} \subset \mathcal{O}_X$ be an ideal sheaf and~$M$ an~$\mathcal{A}$-module. Then
	\begin{equation}
    \mathcal{J} M = 0 \Leftrightarrow (\mathcal{A} \cdot \mathcal{J}) M = 0.
  \end{equation}
\end{lemma}
\begin{proof}
	Easy local computation.
\end{proof}

\begin{definition}
	A sheaf of ideals~$\mathcal{I} \subset \mathcal{A}$ is called \emph{central}, if for any open~$U \subset X$, the ideal~$\mathcal{I}(U) \subset \mathcal{A}(U)$ can be generated by central elements.
\end{definition}

\begin{proposition}
	Let
	\begin{equation}
    0 \to A \to B \to C \to 0
  \end{equation}
	be an exact sequence of coherent~$\mathcal{A}$-modules with~$\Supp(A) = V \subset X$. Then there exists a commutative diagram of~$\mathcal{A}$-modules
	\begin{equation}
	  \begin{tikzcd}
	  	0 \arrow{r} & A \arrow{r} \arrow{d}{\identity} & B \arrow{r} \arrow{d} & C \arrow{d} \arrow{r} & 0 \\
	  	0 \arrow{r} & A \arrow{r} & B' \arrow{r} & C' \arrow{r} & 0
	  \end{tikzcd}
	\end{equation}
	with exact rows and such that~$\Supp(B'), \Supp(C') \subset V$.
	\label{propexseqext}
\end{proposition}
\begin{proof}
	Let~$\mathcal{J} \subset \mathcal{O}_X$ denote the radical ideal corresponding to the closed subset~$V$. Then there exists~$n_0 \in \mathbb{N}$ such that~$\mathcal{J}^n A= 0$ for all~$n \geq n_0$, and by \cref{lmaidealannihilate} it follows that~$(\mathcal{A} \cdot \mathcal{J}^n) A = (\mathcal{A} \cdot \mathcal{J})^n A =0$ all~$n \geq n_0$. For each~$n$, we obtain a commutative diagram with exact rows
	\begin{equation}
	  \begin{tikzcd}
      0 \arrow{r}& A \arrow{r}{\iota} \arrow{d}{\identity} & B \arrow{r}{\pi} \arrow{d}& C \arrow{d} \arrow{r} & 0 \\
	  	& A \arrow{r}{\overline{\iota}} & B/(\mathcal{A} \cdot \mathcal{J})^n B  \arrow{r}{\overline{\pi}}& C/(\mathcal{A} \cdot \mathcal{J})^n C  \arrow{r} & 0
	  \end{tikzcd}
	  \label{eqexseq}
	\end{equation}
	where~$\overline{\iota},\overline{\pi}$ are induced by~$\iota,\pi$ respectively and the non-labeled vertical maps are the canoncial projections. We claim that for~$n$ large enough,~$\overline{\iota}$ is a monomorphism. As we can check injectivity locally, let~$X= \bigcup_{i=1}^r U_i$ with~$U_{i} = \Spec R_i$ open affine. Then, on each~$U_i$, the problem looks as follows: we are given an~$R_i$-algebra~$S_i$, an ideal~$J_i \subset R_i$, an exact of~$S_i$-modules
	\begin{equation}
    0 \to A_i \to B_i \to C_i \to 0
  \end{equation}
	and we know that for all~$n \geq n_i$,~$J^n A =0$. Diagram (\ref{eqexseq}) translates as
	\begin{equation}
    \begin{tikzcd}
      0 \arrow[r] & A_i \arrow[r, "\iota_i"] \arrow[d, equals] & B_i \arrow[r, "\pi_i"] \arrow[d] & C_i \arrow[d] \arrow[r] & 0 \\
      & A_i \arrow[r, "\overline{\iota_i}"] & B_i/(S_i \cdot J_i)^n B_i \arrow[r, "\overline{\pi_i}"] & C/(S_i \cdot J_i)^n C_i \arrow[r] & 0
    \end{tikzcd}
	\end{equation}
	We will now use the Artin--Rees lemma, which is in general not valid for non-commutative rings, but does hold for central ideals like~$S_i \cdot J_i$ (see~\cite[theorem 7.2.1]{MR0231816}): there exists~$q_i \in \mathbb{N}$ such that for all~$m_i \geq q_i$ we have
	\begin{equation}
    A_i \cap (S_i \cdot J_i)^n B_i = (S_i \cdot J_i)^{n-q_i} (A_i \cap (S_i \cdot J_i)^{q_i} B_i).
  \end{equation}
	Now note that~$\ker(\overline{\iota_i}) = A_i \cap (S_i \cdot J_i)^n B_i$, and thus the Artin-Rees lemma tells us that if we choose~$m_i$ such that~$n-q \geq n_i$, then~$\ker(\overline{\iota_i}) = 0$, i.e.~$\overline{\iota_i}$ is injective. Now, if we choose~$n = \max_i m_i$, then~$\overline{\iota_i}$ will be injective for all~$i$, proving that~$\overline{\iota}$ is a monomorphism.

	To conclude the proof, note that for any coherent~$\mathcal{A}$-module~$M$, we have that~$\Supp(M) = \mathrm{V}(\mathrm{Ann}_{\mathcal{O}_X}(M))$ since~$M$ is also~$\mathcal{O}_X$-coherent. But by \cref{lmaidealannihilate}, we know that~$\mathcal{J}^n$ annihilates~$M/(\mathcal{A} \cdot \mathcal{J})^n M = M/(\mathcal{A} \cdot \mathcal{J}^n) M$ as~$\mathcal{A} \cdot \mathcal{J}^n$ does so.  It follows that
	\begin{equation}
    \Supp(B/(\mathcal{A} \cdot \mathcal{J})^n B), \Supp(C/(\mathcal{A} \cdot \mathcal{J})^n C) \subset \mathrm{V}(\mathcal{J}^n) = \mathrm{V}(\mathcal{J}) = V.
  \end{equation}
\end{proof}

\begin{definition}
	For~$p \in \mathbb{Z}$, denote by~$\coh(\mathcal{A})_{\leq p}$ the full subcategory of~$\coh(\mathcal{A})$ consisting of those~$\mathcal{A}$-modules~$M$ with~$\dim(\Supp(M)) \leq p$.
	\label{defabfiltsubcat}
\end{definition}

\begin{remark}
	The properties of~$\Supp(-)$ easily imply that~$\coh(\mathcal{A})_{\leq p}$ is a Serre subcategory of~$\coh(\mathcal{A})_{\leq q}$ if~$p \leq q$.
\end{remark}

\begin{corollary}
  \label{corollary:abelian-derived-filtration}
	The natural functors (see \cref{defabfiltsubcat} and \cref{convsubcatfilt} for the notation)
	\begin{align*}
	\derived^\bounded(\coh(\mathcal{A})_{\leq p}) &\to \derived^\bounded_{\leq p}(\coh(\mathcal{A})) \\
	\derived^\bounded(\coh(\mathcal{A})_{\leq p})/\derived^\bounded(\coh(\mathcal{A})_{\leq p -1}) &\to \derived^\bounded(\coh(\mathcal{A})_{\leq p}/\coh(\mathcal{A})_{\leq p-1})
	\end{align*}
	are equivalences of categories.
	\label{corverdiervsserre}
\end{corollary}
\begin{proof}
  The statement of \cref{propexseqext} is exactly the condition of~\cite[\S1.15, lemma~(c1)]{MR1667558} which makes the above functors equivalences.
\end{proof}

\section{Relative tensor triangular Chow groups of a coherent \texorpdfstring{$\mathcal{O}_X$}{OX}-algebra}
\label{sectionrelgroupsalg}
In this section, we obtain a definition of the relative tensor triangular cycle and Chow groups of a coherent~$\mathcal{O}_X$\dash algebra~$\mathcal{A}$ by means of an action of the derived category of quasi-coherent~$\mathcal{O}_X$-modules~$\derived(\Qcoh(\mathcal{O}_X))$ on the derived category of quasi-coherent right~$\mathcal{A}$-modules~$\derived(\Qcoh(\mathcal{A}))$. We then derive some basic properties of these groups, including a group homomorphism induced by the forgetful functor that relates ~$\CH^{\Delta}_i(X,\mathcal{A})$ to~$\CH_i(X)$ when~$X$ is regular.

The general approach we use for the relative tensor triangular Chow groups works for all quasi-coherent~$\mathcal{O}_X$-algebras~$\mathcal{A}$, but as we will see below, the coherent case will turn out to be more manageable, since then two notions of support will agree for bounded complexes of coherent~$\mathcal{A}$-modules. Therefore we only develop the theory in this setting, which is sufficient for the examples.
\subsection{The action of \texorpdfstring{$\derived(\Qcoh(\mathcal{O}_X))$}{D(Qcoh O\textunderscore X)} on \texorpdfstring{$\derived(\Qcoh(\mathcal{A}))$}{D(Qcoh A)}}
The bifunctor
\begin{equation}
  - \otimes_{\mathcal{O}_X} - : \Qcoh(\mathcal{O}_X) \times \Qcoh(\mathcal{A}) \to \Qcoh(\mathcal{A})
\end{equation}
gives rise to a bifunctor
\begin{equation}
  - \otimes^\LLL_{\mathcal{O}_X} - : \derived(\Qcoh(\mathcal{O}_X)) \times \derived(\Qcoh(\mathcal{A})) \to \derived(\Qcoh(\mathcal{A}))
\end{equation}
by taking~$\mathrm{K}$-flat resolution in the first variable and applying~$- \otimes_{\mathcal{O}_X} -$. This defines an action of~$\derived(\Qcoh(\mathcal{O}_X))$ on~$\derived(\Qcoh(\mathcal{A}))$, where the unitor and associator isomorphisms (\ref{eqnassocunitiso}) are induced by those on the level of complexes, i.e.~the natural isomorphisms
\begin{align*}
	(A^{\bullet} \otimes_{\mathcal{O}_X} B^{\bullet}) \otimes_{\mathcal{O}_X} X^{\bullet} &\xrightarrow{\sim} A^{\bullet} \otimes_{\mathcal{O}_X} (B^{\bullet} \otimes_{\mathcal{O}_X} X^{\bullet})\\
	\mathcal{O}_X \otimes_{\mathcal{O}_X} X^{\bullet} &\xrightarrow{\sim} X^{\bullet}
\end{align*}
for~$A^{\bullet},B^{\bullet}$ complexes of quasi-coherent~$\mathcal{O}_X$-modules and~$X^{\bullet}$ a complex of quasi-coherent right~$\mathcal{A}$-modules.

\begin{remark}
	The action of~$\derived(\Qcoh(\mathcal{O}_X))$ on~$\derived(\Qcoh(\mathcal{A}))$ satisfies the local-to-global principle (see \cref{remlocglob}) since the action~$\derived(\Qcoh(\mathcal{O}_X))$ on itself does so.
\end{remark}


We will now continue to derive some properties of the notion of support that the action of~$\derived(\Qcoh(\mathcal{O}_X))$ on~$\derived(\Qcoh(\mathcal{A}))$ induces on objects of the latter category.
\begin{proposition}
	Let~$V \subset X$ be a specialization-closed subset. Then~$\derived_V(\Qcoh(\mathcal{A}))$ coincides with the subcategory~$\derived(\Qcoh(\mathcal{A}))_V$ of all complexes~$C^{\bullet} \in \derived(\Qcoh(\mathcal{A}))$ such that~$\supp(C^{\bullet}) \subset V$. In particular, the subcategories~$\derived_V(\Qcoh(\mathcal{A}))$ are smashing.
	\label{propsubsetsubcatcoincide}
\end{proposition}
\begin{proof}
	If~$C^{\bullet}$ is a complex of quasi-coherent right~$\mathcal{A}$-modules, then we need to show that~$\supp(C^{\bullet}) \subset V \Leftrightarrow \Supp(C^{\bullet}) \subset V$. If~$X = \bigcup_i U_i$ is an open cover, then it suffices to show that~$\supp(C^{\bullet}) \cap U_i \subset V \cap U_i \Leftrightarrow \Supp(C^{\bullet}) \cap U_i \subset V \cap U_i$ for all~$i$. Let~$U_i = \Spec R_i, i= 1, \ldots, n$ be a cover of~$X$ by affine opens with closed complements~$Z_i$ and set~$V_i \coloneqq U_i \cap V$. Notice that the sets~$V_i$ are still specialization-closed in~$U_i$. We have~$\supp(C^{\bullet}|_{U_i}) = \supp(L_{Z_i}\mathcal{O}_X \ast C^{\bullet}) = \supp(C^{\bullet}) \cap U_i$ by \cref{propstevesuppprop} and~$\Supp(C^{\bullet}|_{U_i}) = \Supp(L_{Z_i}\mathcal{O}_X \ast C^{\bullet}) = \Supp(C^{\bullet}) \cap U_i$ since localization is exact. Hence we have reduced to showing that
	\begin{equation}
    \supp(C^{\bullet}|_{U_i}) \subset V_i \Leftrightarrow \Supp(C^{\bullet}|_{U_i}) \subset V_i$ for~$i =1, \ldots n.
  \end{equation}
  But now, we can assume that~$\mathcal{A}$ is given as an~$R_i$-algebra~$A$ and~$C^{\bullet}|_{U_i}$ a complex of right~$A$-modules. Since both~$\supp$ and~$\Supp$ can be computed by first applying the forgetful functor~$\derived(\Qcoh(\mathcal{A})) \to \derived(\Qcoh(\mathcal{O}_X))$, the result follows from~\cite[proposition~3.14]{MR3594539}, where it is shown that for the complex of~$R_i$-modules~$C^{\bullet}|_{U_i}$, the sets~$\supp(C^{\bullet}|_{U_i})$ and~$\Supp(C^{\bullet}|_{U_i})~$ have the same minimal elements.

	The last statement follows from the first and \cref{lmarelativesmash}.
\end{proof}

Let us show that~$\supp$ and~$\Supp$ coincide for small complexes.
\begin{proposition}
	Let~$C^{\bullet} \in \derived(\Qcoh(\mathcal{A}))$ such that~$\mathrm{H}^{\ast}(C^{\bullet})$ is bounded and coherent. Then~$\supp(C^{\bullet}) = \Supp(C^{\bullet})$.
	\label{propsupp=Supp}
\end{proposition}
\begin{proof}
	As in the proof of \cref{propsubsetsubcatcoincide}, we notice that if~$X = \bigcup_i U_i$ is a cover by affine opens with complements~$Z_i$, then it suffices to show that
	\begin{equation}
    \underbrace{\supp(C^{\bullet}) \cap U_i}_{= \supp(C^{\bullet}|_{U_i})} = \underbrace{\Supp(C^{\bullet}) \cap U_i}_{=\Supp(C^{\bullet}|_{U_i})}
  \end{equation}
  for all~$i$. Hence, we have reduced to the affine case, where the result is implied from the corresponding one for complexes in~$\derived(\Qcoh(\mathcal{O}_X))$. But the latter is well known.
\end{proof}

\begin{remark}
  Given a coherent~$\mathcal{O}_X$\dash algebra~$\mathcal{A}$ (on which~$\mathcal{O}_X$ acts by centrally by assumption) we can consider~$\ZZ(\mathcal{A})$ as a commutative coherent~$\mathcal{O}_X$\dash algebra. Let
  \begin{equation}
    \pi\colon Z\coloneqq\relSpec_X\ZZ(\mathcal{A})\to X
  \end{equation}
  be the relatively affine scheme given by~$\ZZ(\mathcal{A})$. We can consider~$\mathcal{A}$ as a coherent~$\mathcal{O}_Z$\dash algebra, which we will denote~$\mathcal{B}$, and by~\cite[proposition~3.5]{MR3695056} we have that~$\Qcoh_X\mathcal{A}\cong\Qcoh\mathcal{B}$. The action of~$\derived(\Qcoh X)$ and~$\derived(\Qcoh Z)$ will be different in general.

\end{remark}

\subsection{Unwinding the definitions}
\label{subsection:main-result}
With all the technical material we have assembled so far, let us look once more at \cref{definition:1510.00211}. Let~$\mathcal{T} = \derived(\Qcoh(\mathcal{O}_X))$ and~$\mathcal{K} = \derived(\Qcoh(\mathcal{A}))$. Recall that~$X$ is a noetherian separated scheme and~$\mathcal{A}$ is a coherent sheaf of~$\mathcal{O}_X$\dash algebras.
\begin{convention}
	We will write
	\begin{equation}
    \ZZ^{\Delta}_i(X,\mathcal{A}) \quad \text{and} \quad \CH^{\Delta}_i(X,\mathcal{A})
  \end{equation}
	for the groups~$\ZZ^{\Delta}_i(\mathcal{T},\mathcal{K})$ and~$\CH^{\Delta}_i(\mathcal{T},\mathcal{K})$, respectively.
\end{convention}
We have
\begin{equation}
  \ZZ^{\Delta}_{i} (X,\mathcal{A}) = \Kzero\left((\mathcal{K}_{(i)}/\mathcal{K}_{(i-1)})^\compact\right)
\end{equation}
by definition, and both categories~$\mathcal{K}_{(i)},\mathcal{K}_{(i+1)}$ are compactly generated. Hence, we have that
\begin{equation}
  (\mathcal{K}_{(i)}/\mathcal{K}_{(i-1)})^\compact \cong \left((\mathcal{K}_{(i)})^\compact/(\mathcal{K}_{(i-1)})^\compact\right)^{\natural}
\end{equation}
by~\cite[theorem~5.6.1]{MR2681709}. Furthermore,~$(\mathcal{K}_{(i)})^\compact$ coincides with the full subcategory of~$\mathcal{K}^\compact$ consisting of objects with support in codimension~$\geq i$ by~\cite[proposition~2.23]{1510.00211}. From \cref{thmcompgen}, we have
\begin{equation}
  \mathcal{K}^\compact \cong \derived^\perf(\mathcal{A}) \subset \derived^\bounded(\coh(\mathcal{A})).
\end{equation}
Because~$\mathcal{A}$ is assumed to be coherent,~$\Supp$ and~$\supp$ coincide for objects of~$\derived^\bounded(\coh(\mathcal{A}))$ by \cref{propsupp=Supp}. It follows that
\begin{equation}
  \ZZ^{\Delta}_{i} (X,\mathcal{A}) = \Kzero\left(\left(\derived^{\mathrm{perf}}_{\leq i}(\mathcal{A})/ \derived^{\mathrm{perf}}_{\leq i-1}(\mathcal{A})  \right)^{\natural} \right).
\end{equation}
If~$\mathcal{A}$ is additionally of finite global dimension,~$\derived^{\mathrm{perf}}(\mathcal{A}) \cong \derived^\bounded(\coh(\mathcal{A}))$ and we get from \cref{corverdiervsserre} that
\begin{equation}
  \ZZ^{\Delta}_{i} (X,\mathcal{A})= \Kzero\left(\derived^\bounded\left(\coh_{\leq i}(\mathcal{A})/\coh_{\leq i-1}(\mathcal{A})\right)\right) =  \Kzero\left(\coh_{\leq i}(\mathcal{A})/\coh_{\leq i-1}(\mathcal{A})\right).
\end{equation}
Similarly, we deduce in this case an isomorphism of sequences of abelian groups
\begin{equation}
  \begin{gathered}
    \begin{tikzcd}
      \ZZ^{\Delta}_{i} (X,\mathcal{A}) \arrow{r} \arrow[equals]{d}
      & \Kzero\left((\mathcal{K}_{(i+1)}/\mathcal{K}_{(i-1)})^\compact\right) \arrow{r} \arrow[equals]{d}
      & \ZZ^{\Delta}_{i+1} (X,\mathcal{A}) \arrow[equals]{d}
      \\
      \Kzero\left(\coh_{\leq i}(\mathcal{A})/\coh_{\leq i-1}(\mathcal{A})\right) \arrow{r}{\iota}
      & \Kzero\left(\coh_{\leq i+1}(\mathcal{A})/\coh_{\leq i-1}(\mathcal{A})\right) \arrow{r}{\pi}
      & \Kzero\left(\coh_{\leq i+1}(\mathcal{A})/\coh_{\leq i}(\mathcal{A})\right)
    \end{tikzcd}
  \end{gathered}
\end{equation}
which are exact in the middle. Hence, we deduce from \cref{propchowexseq} an isomorphism~$\CH^{\Delta}_{i}(X,\mathcal{A}) \cong \im(\iota) = \ker(\pi)$  for this situation. The lower sequence is the end of the~$\mathrm{K}$-theory long exact localization sequence for the Serre localization
\begin{equation}
  \coh_{\leq i}(\mathcal{A})/\coh_{\leq i-1}(\mathcal{A}) \to \coh_{\leq i+1}(\mathcal{A})/\coh_{\leq i-1}(\mathcal{A}) \to \coh_{\leq i+1}(\mathcal{A})/\coh_{\leq i}(\mathcal{A})
\end{equation}
and hence
\begin{equation}
  \label{equation:chow-as-cokernel}
\CH^{\Delta}_{i} (X,\mathcal{A}) \cong \mathrm{coker}\left(\mathrm{K}_1\left( \coh_{\leq i+1}(\mathcal{A})/\coh_{\leq i}(\mathcal{A}) \right)  \to \Kzero\left( \coh_{\leq i}(\mathcal{A})/\coh_{\leq i-1}(\mathcal{A}) \right)\right).
\end{equation}

There is also a local description of~$\ZZ^{\Delta}_{i} (X,\mathcal{A})$. Abstractly, it follows from~\cite{MR3801492} and~\cite[proposition~2.18, lemma~2.19]{1510.00211}, that
\begin{equation}
\ZZ^{\Delta}_{i} (X,\mathcal{A}) = \coprod_{x \in X_{(i)}} \Kzero\left((\Gamma_x \mathcal{K})^\compact\right),
\label{eqnabstractsplitting}
\end{equation}
where~$X_{i}$ is the set of points~$x \in X$ such that~$\dim(x) = i$.

\begin{lemma}
	Suppose~$\mathcal{A}$ is coherent. Then
	\begin{equation}
    (\Gamma_x \mathcal{K})^\compact \cong \derived_{\lbrace x \rbrace}^\perf(\mathcal{A}_x).
  \end{equation}
	\label{lmasplitcomp}
\end{lemma}
\begin{proof}
	Since for any object~$A \in \mathcal{K}$ we have, by definition,~$\Gamma_x A = \Gamma_{\overline{\lbrace x \rbrace}} L_{Y_x} \mathcal{O}_X \otimes_{\mathcal{O}_X}^\LLL A$, it follows that
	\begin{equation}
    \Gamma_x \mathcal{K} = \Gamma_{\overline{\lbrace x \rbrace}} \mathcal{O}_X \ast \left(L_{Y_x} \mathcal{O}_X \ast \mathcal{K}\right).
  \end{equation}
	The subcategory~$\derived_{Y_x}(\mathcal{A})$ is smashing by \cref{propsubsetsubcatcoincide} and it follows from \cref{lmarelativesmash} and \cref{propstalkquot} that~$L_{Y_x} \mathcal{O}_X \ast \mathcal{K} \cong \derived(\Catmod(\mathcal{A}_x))$. The compact objects of~$\Gamma_x \mathcal{K}$ are given by the compact objects~$a$ of~$L_{Y_x} \mathcal{O}_X \ast \mathcal{K}$ with~$\supp(a) \subset \overline{ \lbrace x \rbrace}$: the inclusion functor~$I: \Gamma_x \mathcal{K} \to L_{Y_x} \mathcal{O}_X \ast \mathcal{K}$ has a coproduct-preserving right adjoint~$\Gamma_{\overline{\lbrace x\rbrace}}(\mathbb{I}) \ast -$ and hence preserves compactness. Thus, the compact objects of~$\Gamma_x \mathcal{K}$ embed into the compact objects of~$L_{Y_x} \mathcal{O}_X \ast \mathcal{K}$ with support in~$\overline{\lbrace x\rbrace}$. On the other hand, if~$a$ is a compact object of~$L_{Y_x} \mathcal{O}_X \ast \mathcal{K}$ with support in~$\overline{\lbrace x\rbrace}$, then the localization triangle
	\begin{equation}
    \Gamma_{\overline{\lbrace x\rbrace}}(\mathbb{I}) \ast a \to a \to L{\overline{\lbrace x\rbrace}}(\mathbb{I}) \ast a \to \Sigma\left(\Gamma_{\overline{\lbrace x\rbrace}}(\mathbb{I}) \ast a\right)
  \end{equation}
	 and \cref{propstevesuppprop} show that~$\Gamma_{\overline{\lbrace x\rbrace}}(\mathbb{I}) \ast a \cong a$, and hence~$a$ belongs to the essential image of the embedding~$I$.

	 Since~$\derived(\Catmod(\mathcal{A}_x))^\compact \cong \derived^\perf(\mathcal{A}_x)$ and~$\supp = \Supp$ for its objects by \cref{propsupp=Supp}, the desired description follows.
\end{proof}

\begin{lemma}
	Let~$(R,\mathfrak{m})$ be a commutative noetherian local ring and~$A$ a (module\dash)finite~$R$-algebra. Then a right~$A$-module~$M$ has finite length over~$A$ if and only if it has finite length over~$R$.
	\label{lmaflsame}
\end{lemma}
\begin{proof}
	Recall that a right module has finite length if and only if it is both artinian and noetherian. Hence, if~$M$ has finite length over~$R$, it must also have finite length over~$A$, since every chain of~$A$-submodules of~$M$ is also a chain of~$R$-submodules.

  In order to prove that right~$A$-modules of finite~$A$-length also have finite~$R$-length, it suffices to show that all simple right~$A$-modules have finite~$R$-length: one can then refine finite composition series over~$A$ to finite composition series over~$R$. In order to study simple right~$A$-modules it suffices to consider simple modules over~$A/\mathrm{J}(A)$, since the Jacobson radical annihilates all simple modules, by definition. We have~$\mathrm{J}(R) = \mathfrak{m}$ and by~\cite[corollary~5.9]{MR1125071}, it follows that~$\mathfrak{m}A \subset \mathrm{J}(A)$, and hence we have a surjection~$A/\mathfrak{m}A \twoheadrightarrow A/\mathrm{J}(A)$. By assumption,~$A/\mathfrak{m}A$ is a finite~$R$-module with support contained in~$\lbrace \mathfrak{m} \rbrace$ and hence has finite length over~$R$. It follows that~$A/\mathrm{J}(A)$ has finite~$R$-length as well. Hence, the finite length right modules over~$A/\mathrm{J}(A)$ have finite length over~$R$, which holds in particular for the simple ones.
\end{proof}

\begin{corollary}
	Suppose~$\mathcal{A}$ is coherent. Then
  \begin{equation}
    \ZZ^{\Delta}_{i} (X,\mathcal{A}) = \coprod_{x \in X_{(i)}} \Kzero\left(\derived_{\fl}^\perf(\mathcal{A}_x)\right).
  \end{equation}
  where~$\derived_{\fl}^\perf(\mathcal{A}_x) \subset \derived^\perf(\mathcal{A}_x)$ denotes the full subcategory of complexes with finite length cohomology. If furthermore~$\mathcal{A}$ has finite global dimension, then
	\begin{equation}
    \ZZ^{\Delta}_{i} (X,\mathcal{A}) = \coprod_{x \in X_{(i)}} \Kzero\left(\derived^\bounded(\fl\mathcal{A}_x)\right),
  \end{equation}
	where~$\fl\mathcal{A}_x$ denotes the abelian category of right~$\mathcal{A}_x$-modules of finite length.
	\label{corsplitcycles}
\end{corollary}
\begin{proof}
  For the first statement, it suffices to prove that~$\derived^\perf_{\lbrace x \rbrace}(\mathcal{A}_x)  \cong \derived_{\fl}^\perf(\mathcal{A}_x)$ by \cref{lmasplitcomp}. This follows from \cref{lmaflsame} since a complex~$C^{\bullet} \in \derived^\perf(\mathcal{A}_x)$ has support in~$\lbrace x \rbrace$ if and only if~$\Supp\HH^\bullet(C^{\bullet})) \subset \lbrace x \rbrace$ if and only if~$\HH^\bullet(C^{\bullet})$ has finite~$\mathcal{O}_{X,x}$-length if and only if~$\HH^\bullet(C^{\bullet})$ has finite~$\mathcal{A}_x$-length.

	For the second assertion, \cref{corverdiervsserre} gives
	\begin{equation}
    \derived^\perf_{\lbrace x \rbrace}(\mathcal{A}_x) \cong \derived^\bounded_{\lbrace x \rbrace}(\catmod(\mathcal{A}_x)) \cong \derived^\bounded(\catmod(\mathcal{A}_x)_{\lbrace x \rbrace})
  \end{equation}
	Now a finitely generated right~$A_{x}$-module has support in~$\lbrace x \rbrace$ if and only if it has finite length as an~$R$-modules if and only if it has finite length as a right~$\mathcal{A}_x$-module by \cref{lmaflsame}. This shows that~$\catmod(\mathcal{A}_x)_{\lbrace x \rbrace} \cong \fl\mathcal{A}_x$ and finishes the proof.
\end{proof}

\Cref{corsplitcycles} makes it possible to give a computation of~$\ZZ_i^{\Delta}(X,\mathcal{A})$ in large generality.
\begin{theorem}
  \label{theorem:cycle-groups}
	Let~$X$ be a noetherian scheme and~$\mathcal{A}$ a coherent~$\mathcal{O}_X$-algebra of finite global dimension. Then
	\begin{equation}
    \ZZ^{\Delta}_{i} (X,\mathcal{A}) = \bigoplus_{x \in X_{(i)}} \mathbb{Z}^{r_x}
  \end{equation}
	where~$r_x < \infty$ is the number of isomorphism classes of simple right modules of~$\mathcal{A}_x$.
\end{theorem}
\begin{proof}
	By \cref{corsplitcycles}, it suffices to show that~$\Kzero(\derived^\bounded(\fl\mathcal{A}_x) = \Kzero(\fl\mathcal{A}_x)=\mathbb{Z}^{r_x}$ with~$r_x < \infty$. From the proof of \cref{lmaflsame} we see that the simple~$\mathcal{A}_x$-modules correspond to the simple~$\mathcal{A}_x/\mathrm{J}(\mathcal{A}_x)$-modules, and that the latter algebra is of finite length over~$\mathcal{O}_{X,x}$. This implies that~$\mathcal{A}_x/\mathrm{J}(\mathcal{A}_x)$ is right Artinian and hence has~$r_x < \infty$ simple right modules (all of them occur in a composition series of~$\mathcal{A}_x$ over itself by the Jordan-H\"older theorem). A standard induction on the composition multiplicities of these simple modules shows that~$\Kzero(\fl\mathcal{A}_x)=\mathbb{Z}^{r_x}$ as desired.
\end{proof}

Let us finish the section with an easy observation concerning the vanishing of~$\ZZ^{\Delta}_i(X,\mathcal{A})$ and~$\CH^{\Delta}_i(X,\mathcal{A})$.
\begin{proposition}
  \label{proposition:vanishing-outside-range}
	Suppose~$\dim(\supp(\mathcal{A})) = n$. Then
	\begin{equation}
    \ZZ^{\Delta}_i (X,\mathcal{A})= \CH^{\Delta}_i(X,\mathcal{A}) = 0
  \end{equation}
	for all~$i > n$.
\end{proposition}
\begin{proof}
	If~$i>n$, then~$\mathcal{K}_{i} = \mathcal{K}_{i-1} = \mathcal{K}$ and hence
	\begin{equation}
    \ZZ^{\Delta}_{i} (X,\mathcal{A})= \Kzero\left((\mathcal{K}_{(i)}/\mathcal{K}_{(i-1)})^\compact\right) = 0,
  \end{equation}
	which also implies~$\CH^{\Delta}_i(X,\mathcal{A}) = 0$.
\end{proof}

\subsection{Comparison to Chow groups of \texorpdfstring{$X$}{X} for coherent \texorpdfstring{$\mathcal{O}_X$}{OX}-algebras on regular schemes}
Suppose that~$\mathcal{A}$ is a coherent~$\mathcal{O}_X$-algebra and that~$X$ is regular. By definition of~$\supp$, the forgetful functor~$U: \derived(\Qcoh(\mathcal{A})) \to \derived(\Qcoh(\mathcal{O}_X))$ induces functors
\begin{equation}
  \derived(\Qcoh(\mathcal{A}))_{(p)} \to \derived(\Qcoh(\mathcal{O}_X))_{(p)}
\end{equation}
for all~$p \geq 0$. If~$C^{\bullet}$ is a perfect complex in~$\derived(\Qcoh(\mathcal{A}))$, then~$U(C^{\bullet})$ will be an object of~$\derived^\bounded(\coh(X)) = \derived^\perf(X)$ and hence~$U$ preserves compactness. Hence, we obtain a commutative diagram of functors
\begin{equation}
	\begin{tikzcd}
		\derived(\Qcoh(\mathcal{A}))_{(p)}^\compact \arrow{r} \arrow{d} \arrow{dr}& \underbrace{(\derived(\Qcoh(\mathcal{A}))_{(p)}^\compact/\derived(\Qcoh(\mathcal{A}))_{(p-1)}^\compact)^{\natural}}_{= (\derived(\Qcoh(\mathcal{A}))_{(p)}/\derived(\Qcoh(\mathcal{A}))_{(p-1)})^\compact} \arrow{dr} &\\
		\derived(\Qcoh(\mathcal{A}))_{(p+1)}^\compact \arrow{dr} & \derived(\Qcoh(\mathcal{O}_X))_{(p)}^\compact \arrow{r} \arrow{d}& \underbrace{(\derived(\Qcoh(\mathcal{O}_X))_{(p)}^\compact/\derived(\Qcoh(\mathcal{O}_X))_{(p-1)}^\compact)^{\natural}}_{= (\derived(\Qcoh(\mathcal{O}_X))_{(p)}/\derived(\Qcoh(\mathcal{O}_X))_{(p-1)})^\compact}\\
		& \derived(\Qcoh(\mathcal{O}_X))_{(p+1)}^\compact &
	\end{tikzcd}
	\label{eqncomptobase}
\end{equation}
in which the horizontal arrows are given by the Verdier quotient followed by the inclusion into the idempotent completion, the vertical arrows are inclusions and the diagonal ones are induced by~$U$.
\begin{remark}
	It is possible to construct the above diagram without assuming~$X$ to be regular: the main obstruction is for~$U$ to preserve compactness. This happens for example, when~$U$ admits a coproduct-preserving right adjoint. But the functor~$\RRR\sheafHom_{\mathcal{O}_X}(U(\mathcal{A}), -)$ is always right adjoint to~$U$. It will preserve coproducts if~$U(\mathcal{A})$ is a perfect complex over~$X$ by~\cite[proof right after Example 1.13]{MR1308405}. Hence, we see that, instead of assuming that~$X$ is regular, it suffices that~$U(\mathcal{A})$ is perfect. If~$X$ is regular this is, of course, always the case.
	\label{remcomparisonsingular}
\end{remark}

\begin{proposition}
	Suppose that~$\mathcal{A}$ is a coherent~$\mathcal{O}_X$-algebra on a noetherian regular scheme~$X$. Let~$\mathcal{T} = \derived(\Qcoh(\mathcal{O}_X))$ and~$\mathcal{K} = \derived(\Qcoh(\mathcal{A}))$. Then the forgetful functor~$\derived(\Qcoh(\mathcal{A}))_{(p)} \to \derived(\Qcoh(\mathcal{O}_X))_{(p)}$ induces group homomorphisms
	\begin{equation}
    \ZZ^{\Delta}_{p} (X,\mathcal{A}) \to \ZZ^{\Delta}_{p} (X,\mathcal{O}_X) = \ZZ_p(X) \quad \text{and} \quad \CH^{\Delta}_{p} (X,\mathcal{A}) \to \CH^{\Delta}_{p} (X,\mathcal{O}_X) = \CH_p(X)
  \end{equation}
	for all~$p \geq 0$.
	\label{propcomparisonmap}
\end{proposition}
\begin{proof}
	This follows immediately from \cref{thmchowrecover} and the definitions of~$\ZZ^{\Delta}_{p}(X,\mathcal{A})$ and~$\CH^{\Delta}_{p} (X,\mathcal{A})$ by applying~$\Kzero(-)$ to~(\ref{eqncomptobase}).
\end{proof}

\begin{remark}
	If in \cref{propcomparisonmap} we only assume that~$U(\mathcal{A})$ is perfect instead of~$X$ being regular (see \cref{remcomparisonsingular}), then~$U$ still gives group homomorphisms
	\begin{equation}
    \ZZ^{\Delta}_{p} (X,\mathcal{A}) \to \ZZ^{\Delta}_{p} (X,\mathcal{O}_X) \quad \text{and} \quad \CH^{\Delta}_{p} (X,\mathcal{A}) \to \CH^{\Delta}_{p} (X,\mathcal{O}_X)
  \end{equation}
	for all~$p \geq 0$.
\end{remark}

\begin{remark}
	Extension of scalars~$\mathcal{A} \otimes^\LLL_{\mathcal{O}_X} -$ has a a coproduct-preserving right adjoint~$U$ and hence preserves compact object. For~$C^{\bullet} \in \derived^\perf(X)$, we have
	\begin{equation}
    \supp(\mathcal{A}  \otimes^\LLL_{\mathcal{O}_X} C^{\bullet}) = \supp(\mathcal{A})  \cap \supp(C^{\bullet})
  \end{equation}
	from which we deduce that~$\mathcal{A} \otimes^\LLL_{\mathcal{O}_X} -$ restricts to
	\begin{equation}
    \derived^\perf(\mathcal{O}_X)_{(p)} \to \derived^\perf(\mathcal{A})_{(p)}
  \end{equation}
	for all~$p \geq 0$. Hence, by a similar argument as for~$U$, we obtain that extension of scalars induces morphisms~$\CH^{\Delta}_{p} (X,\mathcal{O}_X) \to \CH^{\Delta}_{p} (X,\mathcal{A})$. Note however, that if~$\dim(\supp(\mathcal{A})) = q$, then these morphisms are necessarily trivial for~$p >q$ since~$\ZZ^{\Delta}_{p} (X,\mathcal{A}) = \CH^{\Delta}_{p} (X,\mathcal{A}) = 0$ in this case by \cref{proposition:vanishing-outside-range}.
\end{remark}

\section{The case of coherent commutative \texorpdfstring{$\mathcal{O}_X$}{OX}-algebras}
\label{sectioncommcohalg}
In the following, we will show, how the framework we have set up lets us deal with finite morphisms between noetherian schemes. Let~$X$ be a noetherian separated scheme and~$\mathcal{A}$ a \emph{commutative}~$\mathcal{O}_X$-algebra which is \emph{coherent} as an~$\mathcal{O}_X$-module. Then~$\mathcal{A}$ corresponds to an affine morphism~$\varphi: Y \coloneqq \relSpec\mathcal{A} \to X$ and there is an equivalence of categories~$\Theta:\Qcoh(\mathcal{A}) \cong \Qcoh(\mathcal{O}_Y)$ that makes the following diagram commute up to natural isomorphism:

\begin{equation}
\begin{tikzcd}
  \Qcoh(\mathcal{A}) \arrow{rr}{\Theta}  \arrow{rd}{U} & & \Qcoh(\mathcal{O}_Y) \arrow{ld}{\varphi_*} \\
& \Qcoh(\mathcal{O}_X)&
\end{tikzcd}
\label{eqncoherentcommutative}
\end{equation}

Let us note that~$\Theta$ also restricts to an equivalence between the subcategories of coherent modules and the restriction makes a diagram similar to (\ref{eqncoherentcommutative}) commute, with~$\Qcoh(-)$ replaced by~$\coh(-)$. The following three results should be well-known.

\begin{lemma}
	The morphism~$\varphi$ is finite. In particular,~$Y$ is noetherian and separated.
	\label{lmacohalgfinmor}
\end{lemma}
\begin{proof}
	This is an immediate consequence of the construction of~$\relSpec\mathcal{A}$: over each open affine~$U = \Spec R$ of~$X$ lies an open affine~$\Spec\mathcal{A}(U)$, and~$\mathcal{A}(U)$ is a finite~$R$-module since~$\mathcal{A}$ was assumed to be a coherent sheaf on~$X$.
\end{proof}

\begin{lemma}
	Let~$f: Y \to X$ be a morphism of schemes and assume~$X$ locally noetherian.
	\begin{enumerate}
		\item For any coherent~$\mathcal{O}_X$-module~$M$, we have~$\Supp(f^*(M)) = f^{-1}(\Supp(M))$.
		\item Suppose~$f$ is finite. For any closed subset~$Z \subset \im(f)$, we have
		\begin{equation}
      \dim(f^{-1}(Z)) = \dim(Z)
    \end{equation}
		and for any closed set~$W \subset Y$, we have
		\begin{equation}
      \dim(f(W)) = \dim(W)
    \end{equation}
	\end{enumerate}
	\label{lmadimfinmor}
\end{lemma}
\begin{proof}[Sketch of the proof]
	For the first assertion we can assume that~$X,Y$ are affine, in this case the statement is proved in \cite[exercise~3.19(viii)]{MR242802}. For the second statement, we consider the fibre square
		\begin{equation}
		  \begin{tikzcd}
        f^{-1}Z \arrow{r} \arrow{d} & Y \arrow{d}{f} \\
		  	Z \arrow{r} & \im(f)
		  \end{tikzcd}
		\end{equation}
		and use that for finite and surjective morphisms, domain and codomain have the same Krull dimension. The last assertion follows from the second one by considering the composition~$f|_W: W \to Y \xrightarrow{f} X$.
\end{proof}

%
%

\begin{proposition}
	Suppose~$X$ is a locally noetherian scheme and~$f: X \to Y$ is an affine closed morphism and~$M$ a quasi-coherent~$\mathcal{O}_X$-module. Then~$\Supp(f_*M) = f(\Supp(M))$.
	\label{propsuppdirectcommute}
\end{proposition}
\begin{proof}[Sketch of the proof]
		We shall compute the stalks of the sheaf~$f_*M$ at~$y \in Y$. Since~$f$ is closed, this can be done using all opens on~$X$, i.e.~$(f_*M)_y = \varinjlim_{V \supset f^{-1}(y)} M(V)$. The set~$f^{-1}(y)$ will be contained in an affine open~$\Spec R\subset X$ because~$f$ is affine and hence, we can assume that~$M$ is an~$R$-module and~$f^{-1}(y) =: P$ is a set of prime ideals of~$R$. We rewrite
		\begin{equation}
      (f_*M)_y = \varinjlim_{V \supset f^{-1}(y)} M(V) = \varinjlim_{D(r) \supset P} M_r,
    \end{equation}
		where~$D(r)$ runs over the basic opens of~$\Spec(R)$ that contain~$P$. From this, we see that~$(f_*M)_y = S^{-1}M$, where~$S \coloneqq R \setminus \bigcup_{\mathfrak{p} \in P} \mathfrak{p}$. It follows that~$(f_*M)_y = 0$ if and only if~$M_{\mathfrak{p}} = 0$ for all~$\mathfrak{p} \in P = f^{-1}(y)$, which proves the claim.
\end{proof}

\begin{corollary}
	The equivalence~$\Theta: \Qcoh(\mathcal{A}) \to \Qcoh(\mathcal{O}_Y)$ respects dimension of support: if~$M \in \Qcoh(\mathcal{A})$, then~$\dim(\Supp_X(M)) = \dim(\Supp_Y(\Theta(M)))~$. Hence,~$\Theta$ induces exact equivalences
	\begin{equation}
    \Qcoh(\mathcal{A})_{\leq p} \xrightarrow{\sim} \Qcoh(\mathcal{O}_Y)_{\leq p}
  \end{equation}
	for all~$p \geq 0$.
	\label{corfinitedimpres}
\end{corollary}
\begin{proof}
	By definition and (\ref{eqncoherentcommutative}), we have
	\begin{equation}
    \dim(\Supp(M)) = \dim(\Supp_X(U(M))) = \dim(\Supp_X(\varphi_*(\Theta(M)))).
  \end{equation}
	Since~$\mathcal{A}$ was assumed to be coherent,~$\varphi$ is finite by \cref{lmacohalgfinmor} and it follows from \cref{propsuppdirectcommute} that
	\begin{equation}
    \dim(\Supp_X(\varphi_*(\Theta(M)))) = \dim(\varphi(\Supp_X(\Theta(M))))
  \end{equation}
	as finite morphisms are in particular affine and (universally) closed. By \cref{lmadimfinmor}, the latter quantity is equal to~$\dim(\Supp_Y(\Theta(M)))$, which proves the claim.
\end{proof}

\begin{corollary}
	The functor~$\Theta$ induces an equivalence
	\begin{equation}
    \derived(\Qcoh(\mathcal{A}))_{(p)} \cong \derived(\Qcoh(\mathcal{O}_Y))_{(p)}
  \end{equation}
	for all~$p \geq 0$.
	\label{corthetasubcatequiv}
\end{corollary}
\begin{proof}
	The equivalence~$\Theta$ is exact (as any equivalence of abelian categories) and hence induces and equivalence~$\derived(\Qcoh(\mathcal{A})) \cong \derived(\Qcoh(\mathcal{O}_Y))$. Now, it suffices to remark that for~$C^{\bullet} \in \derived(\Qcoh(\mathcal{A}))$ we have
	\begin{align*}
	C^{\bullet} \in \derived(\Qcoh(\mathcal{A}))_{(p)} &\Leftrightarrow \HH^\bullet(C^{\bullet}) \in \Qcoh(\mathcal{A})_{\leq p} \\
	&\Leftrightarrow \HH^\bullet(\Theta(C^{\bullet})) \in \Qcoh(\mathcal{O}_Y)_{\leq p}  \\
	&\Leftrightarrow \Theta(C^{\bullet}) \in \derived(\Qcoh(\mathcal{O}_X))_{(p)}
	\end{align*}
	where we used \cref{propsubsetsubcatcoincide} and \cref{corfinitedimpres}.
\end{proof}

\begin{theorem}
	Let~$X$ be a separated scheme of finite type over a field and~$\mathcal{A}$ a coherent sheaf of \emph{commutative}~$\mathcal{O}_X$-algebras. Then
	\begin{equation}
    \CH^{\Delta}_{p}(X,\mathcal{A}) \cong \CH^{\Delta}_{p}(Y,\mathcal{O}_Y)
  \end{equation}
	for all~$p \geq 0$. In particular if~$\relSpec\mathcal{A}$ is regular ($\Leftrightarrow \mathcal{A}$ has finite global dimension), then
	\begin{equation}
    \CH^{\Delta}_{p}(X,\mathcal{A}) \cong \CH_p(\relSpec\mathcal{A}).
  \end{equation}
	\label{thmcommrecover}
\end{theorem}
\begin{proof}
	There is a diagram
	\begin{equation}
    \begin{tikzcd}
	  	\Kzero\left((\derived(\Qcoh(\mathcal{A}))_{(p)})^\compact\right) \arrow[r] \arrow[dr] \arrow[d] & \Kzero\left((\derived(\Qcoh(\mathcal{A}))_{(p)}/\derived(\Qcoh(\mathcal{A}))_{(p+1)})^\compact\right)	\arrow[dr] & \\
	  	\Kzero\left((\derived(\Qcoh(\mathcal{A}))_{(p-1)})^\compact\right)  \arrow[dr] & \Kzero\left((\derived(\mathcal{O}_Y)_{(p)})^\compact\right) \arrow[r] \arrow[d]& \Kzero\left((\derived(\mathcal{O}_Y)_{(p)}/\derived(\mathcal{O}_Y)_{(p+1)})^\compact\right) \\
	  	& \Kzero\left((\derived(\mathcal{O}_Y)_{(p-1)})^\compact\right)
    \end{tikzcd}
	\end{equation}
	where all diagonal arrows are isomorphisms induced by~$\Theta$, as follows from \cref{corthetasubcatequiv}. This immediately gives the desired isomorphisms of Chow groups.
\end{proof}

\section{Relative tensor triangular Chow groups for orders}
\label{sectionorders}
In this section we study relative tensor triangular Chow groups for a special class of coherent~$\mathcal{O}_X$\dash algebras: orders. These are particularly well-behaved \emph{noncommutative} algebras, whose definition we recall in \cref{subsection:orders-preliminaries}. In their modern incarnation they were defined in~\cite{MR0117252} and the main reference is~\cite{MR0393100}. The main goal is to show that they coincide with other invariants in the literature, as is the case in the commutative setting where tensor triangular Chow groups agree with the classical Chow groups, see~\cite{MR3423452,1510.00211}.

We give some general results on cycle groups in \cref{subsection:cycle-groups}, based on \cref{theorem:cycle-groups}. We get a description of the top degree cycle groups for any order in \cref{proposition:highest-cycle-group}.

Finally we will use the structure theory for hereditary orders over discrete valuation rings to describe all cycle groups of hereditary orders and the codimension one cycle groups of tame orders, making the result in \cref{theorem:cycle-groups} concrete in a well-known example.

In \cref{subsection:chow-groups} we discuss Chow groups for orders. An easy corollary of the theory is a description of the top degree Chow group in \cref{proposition:highest-chow-group}. More importantly, we recall the definition of various class groups in the theory of orders, and we show that these classical invariants agree with the appropriate tensor triangular Chow groups.

In \cref{subsection:group-rings} we study the Chow groups of group rings over Dedekind domains, for which it is again possible to relate the tensor triangular Chow groups to classical invariants. We give some explicit examples on how one can compute them for integral group rings, using tools from algebraic number theory and representation theory.

\subsection{Preliminaries on orders}
\label{subsection:orders-preliminaries}
In this section we will introduce some basic results about orders on schemes. There are no new results here, but the literature at this level of generality is somewhat scattered.

Observe that for most of this section we will assume that we are working in a central simple algebra. This corresponds to the more geometric approach to the theory of orders. In \cref{subsection:group-rings} we will relax this condition, and consider algebras which are only separable over the generic point, as is common in representation theory and algebraic number theory. We will explain how the results of \cref{subsection:cycle-groups,subsection:chow-groups} change in this more general situation.
\begin{definition}
  \label{definition:order}
  Let~$X$ be an integral normal noetherian scheme with function field~$K$. Let~$A_K$ be a central simple~$K$\dash algebra. An \emph{$\mathcal{O}_X$\dash order}~$\mathcal{A}$ in~$A_K$ is a torsion-free coherent~$\mathcal{O}_X$\dash algebra whose generic fibre is~$A_K$.

  We say that~$\mathcal{A}$ is a \emph{maximal order} if it is not properly contained in another order.
\end{definition}

In~\cite{MR0393100} (maximal) orders are studied in both the geometric and arithmetic setting, mostly in the case of dimension~1. The behaviour of orders in higher dimension quickly becomes more and more complicated.

We will need two more classes of orders, besides just the maximal ones. Recall that Auslander--Goldman characterized maximal orders as those orders which are reflexive as~$\mathcal{O}_X$\dash modules, and for which~$\mathcal{A}_{\eta_Y}$ is a maximal order over the discrete valuation ring~$\mathcal{O}_{X,\eta_Y}$, for all~$\eta_Y$ a point of codimension~1. In dimension one there is a larger class of orders whose behaviour is as nice as that of the maximal orders, and of which maximal orders are a special instance.
\begin{definition}
  Assume that~$X$ is regular and of dimension~1. Then we say that~$\mathcal{A}$ is an \emph{hereditary order} if~$\mathcal{A}(U)$ is of global dimension~1 for every affine open~$U\subseteq X$.
\end{definition}

For hereditary (and maximal) orders in dimension~1 there exists an extensive structure theory. Inspired by the Auslander--Goldman maximality criterion we can introduce a final class of orders, for which one can bootstrap the structure theory of hereditary orders.
\begin{definition}
  We say that~$\mathcal{A}$ is a \emph{tame order} if it is reflexive as an~$\mathcal{O}_X$\dash module, and~$\mathcal{A}_{\eta_Y}$ is an hereditary order over the discrete valuation ring~$\mathcal{O}_{X,\eta_Y}$, for all~$\eta_Y$ a point of codimension~1.
\end{definition}
The notion of tame generalises hereditary orders to higher dimensions.

We now give some examples of orders for which we can describe the tensor triangular cycle and Chow groups.

\begin{example}
  The easiest examples of maximal orders are matrix algebras and their \'etale twisted forms: Azumaya algebras.
\end{example}

\begin{example}
  An example of an hereditary but non-maximal order on~$\mathbb{P}_k^1$ is
  \begin{equation}
    \mathcal{A}\coloneqq
    \begin{pmatrix}
      \mathcal{O}_{\mathbb{P}_k^1} & \mathcal{O}_{\mathbb{P}_k^1} \\
      \mathcal{O}_{\mathbb{P}_k^1}(-p) & \mathcal{O}_{\mathbb{P}_k^1}
    \end{pmatrix}
  \end{equation}
  where~$p\in\mathbb{P}_k^1$ is a closed point. The algebra structure is induced from the embedding in~$\Mat_2(\mathcal{O}_{\mathbb{P}_k^1})$.

  For each closed point~$q\neq p$ we see that~$\mathcal{A}_q$ is isomorphic to the matrix ring over~$\mathcal{O}_{\mathbb{P}_k^1,q}$, whereas for the point~$p$ we get the non-maximal order
  \begin{equation}
    \mathcal{A}_p\cong
    \begin{pmatrix}
      \mathcal{O}_{\mathbb{P}_k^1,p} & \mathcal{O}_{\mathbb{P}_k^1,p} \\
      \mathfrak{m} & \mathcal{O}_{\mathbb{P}_k^1,p}
    \end{pmatrix}.
  \end{equation}
  It is precisely this non-maximality that will contribute to the structure of the relative Chow group, see \cref{corollary:quasiprojective-curve-matrix}.
\end{example}

\subsection{Cycle groups}
\label{subsection:cycle-groups}
Using \cref{theorem:cycle-groups} we have a complete description of cycle groups of coherent~$\mathcal{O}_X$\dash algebras. In this section we discuss what happens in the special case of orders. First we observe that the top-dimensional Chow group always is of the same form.
\begin{proposition}
  \label{proposition:highest-cycle-group}
  Let~$X$ be an integral normal noetherian scheme of dimension~$n$. Let~$\mathcal{A}$ be an order on~$X$. Then
  \begin{equation}
    \ZZ_n^\Delta(X,\mathcal{A})\cong\mathbb{Z}.
  \end{equation}

  \begin{proof}
    Let~$\eta$ be the unique generic point of~$X$. Then~$\mathcal{A}_\eta$ is a central simple algebra over the function field~$\mathcal{O}_{X,\eta}$ and by Morita theory we can conclude from \cref{theorem:cycle-groups}, as there is a unique simple for a division algebra.
  \end{proof}
\end{proposition}

There are several issues in computing the cycle and Chow groups for orders in other degrees:
\begin{enumerate}
  \item there is no general structure theory for (maximal) orders on local rings in arbitrary dimension;
  \item even if there is such a description (as will be the case in dimension~1) the non-splitness of the central simple algebra over the generic point will play an important role, because the higher K-theory of central simple algebras (let alone orders) is different in general from the K-theory of the center.
\end{enumerate}

Nevertheless, in the one-dimensional case we can obtain an explicit description.

First we consider the complete local case, for which there exists an explicit description of hereditary orders~\cite[\S39]{MR0393100}. In this affine situation we will use ring-theoretical notation from op.~cit. In particular, we consider a (complete) discrete valuation ring~$(R,\mathfrak{m})$ whose field of fractions is denoted~$K$, and an hereditary~$R$\dash order~$\Lambda$ in a central simple~$K$\dash algebra~$A\cong\Mat_n(D)$, where~$D$ is a division algebra over~$K$. Then there exists a unique maximal~$R$\dash order~$\Delta$ in~$D$, and we have a block decomposition
\begin{equation}
  \Lambda=
  \begin{pmatrix}
    \Delta & \rad\Delta & \rad\Delta & \ldots & \rad\Delta \\
    \Delta & \Delta     & \rad\Delta & \ldots & \rad\Delta \\
    \Delta & \Delta     & \Delta     & \ldots & \rad\Delta \\
    \ldots &            &            &        & \ldots \\
    \Delta & \Delta     & \Delta     & \ldots & \Delta \\
  \end{pmatrix}^{n_1,\ldots,n_r}
\end{equation}
where the block decomposition is given by putting~$\Mat_{n_i\times n_j}(\Delta)$ (resp.~$\Mat_{n_i\times n_j}(\rad\Delta)$) if~$i\geq j$ (resp.~$i<j$). In particular,~$\sum_{i=1}^rn_i=n$.
\begin{definition}
  The number of blocks~$r$ in the block decomposition is the \emph{type} of~$\Lambda$.
\end{definition}

The following result can be proved along the same lines as \cref{theorem:quasiprojective-curve}, but we give an alternative proof here using d\'evissage in algebraic K-theory~\cite[\S5]{MR0338129}.
\begin{proposition}
  \label{proposition:cDVR-type}
  Let~$R$ be a complete discrete valuation ring, with fraction field~$K$ and residue field~$k$. Let~$\Lambda$ be an hereditary~$R$\dash order in the central simple~$K$\dash algebra~$A$. Then
  \begin{equation}
    \ZZ_0^\Delta(R,\Lambda)\cong\mathbb{Z}^r
  \end{equation}
  where~$r$ is the type of~$\Lambda$.

  \begin{proof}
    By d\'evissage for algebraic K-theory and the invariance of K-theory under nilpotent thickenings applied to~\cite[corollary~39.18(iii)]{MR0393100} we have that 
    \begin{equation}
      \Kzero(\fl\Lambda)\cong\Kzero(\Lambda/\rad\Lambda).
    \end{equation}
    By~\cite[(39.17)]{MR0393100} we have
    \begin{equation}
      \Kzero(\Lambda/\rad\Lambda)\cong\bigoplus_{i=1}^r\Kzero(\Mat_{n_i}(\Delta/\rad\Delta))\cong\mathbb{Z}^{\oplus r}
    \end{equation}
    where~$\Delta/\rad\Delta$ is a skew field over~$k$ by \cite[corollary~17.5]{MR0393100}.

    Similarly one can by d\'evissage appeal to~\cite[corollary~39.18(v)]{MR0393100} for the conclusion.
  \end{proof}
\end{proposition}

\subsection{Chow groups in the regular case}
\label{subsection:chow-groups}
In this section we prove the main results for orders: \cref{corollary:reduced-projective-class-group-affine} shows that for an hereditary order over a Dedekind domain the 0th relative Chow group agrees with the reduced projective class group, and if the order is moreover maximal it agrees with the ideal class group. These are classical invariants that will be introduced below. In the setting of a quasiprojective curve over a field we get the analogous result in \cref{corollary:reduced-projective-class-group-projective}, from which we obtain \cref{theorem:quasiprojective-curve}.

As an immediate corollary to \cref{proposition:highest-cycle-group} and the description of the rational equivalence we have the following general result.
\begin{proposition}
  \label{proposition:highest-chow-group}
  With notation and assumptions as in \cref{proposition:highest-cycle-group} we have that
  \begin{equation}
    \CH_n^\Delta(X,\mathcal{A})\cong\mathbb{Z}.
  \end{equation}

  \begin{proof}
    We have that~$q^\natural(\ker(i))$ from \eqref{eqncommdiagsubcats} is zero because~$i$ is an isomorphism if~$p\geq n$.
  \end{proof}
\end{proposition}
A similar proof of course works for every coherent~$\mathcal{O}_X$\dash algebra, where the cycle group is given by the Grothendieck group of a certain finite-dimensional algebra over the function field, in particular it is easy to construct examples for which
\begin{equation}
  \CH_n^\Delta(X,\mathcal{A})\neq\mathbb{Z},
\end{equation}
e.g.\ by taking~$\mathcal{A}=\mathcal{O}_X\oplus\mathcal{O}_X$.

\paragraph{Classical invariants}
In the~1\dash dimensional case the only other tensor triangular Chow group we need to describe is~$\CH_0^\Delta$, see \cref{proposition:vanishing-outside-range}. We will do this using \cref{propchowexseq}, which allows us to interpret the tensor triangular Chow groups in terms of classical invariants such as the ideal class group and the reduced projective class group, whose definitions we now recall in the affine setting.

Let~$R$ be a Dedekind domain, and denote its quotient field by~$K$. Let~$\Lambda$ be an~$R$\dash order in a central simple~$K$\dash algebra~$A$. Let~$M,N$ be left~$\Lambda$\dash modules. We say that they are \emph{stably isomorphic} if there exists an integer~$r$ and an isomorphism~$M\oplus\Lambda^{\oplus r}\cong N\oplus\Lambda^{\oplus r}$.
\begin{definition}
  The \emph{ideal class group} (or \emph{stable class group})~$\Cl\Lambda$ of~$\Lambda$ consists of the stable isomorphism classes of left~$\Lambda$\dash ideals (i.e.~those submodules~$I$ such that~$KI=A$), where the group structure is defined in~\cite[theorem~35.5]{MR0393100}.
\end{definition}
It is a one-sided generalisation of the usual class group (or Picard group). There also exists a two-sided version, which is different in general, see \cref{remark:one-vs-two-sided}. Because we are only considering the module structure on one side, it is the former and not the latter that is important to us.

In this case the localisation sequence that is used to define rational equivalence in the zeroth Chow group as in \eqref{equation:chow-as-cokernel} is also known as the \emph{Bass--Tate sequence}~\cite{MR712062,MR925271}. We will now recall the description from~\cite[\S2]{MR0404410}. In the relevant degrees the localization sequence takes on the form
\begin{equation}
  \label{equation:bass-tate}
  \Kone(\Lambda)\to\Kone(A)\to\Kzero(\fl\Lambda)\to\Kzero(\Lambda)\to\Kzero(A)\to 0.
\end{equation}
We can also apply d\'evissage to the term~$\Kzero(\fl\Lambda)$, and obtain
\begin{equation}
  \Kzero(\fl\Lambda)\cong\bigoplus_{\mathfrak{p}\in\Spec R\setminus\{0\}}\Kzero(\fl\Lambda_{\mathfrak{p}}).
\end{equation}

\begin{definition}
  The \emph{reduced projective class group}~$\widetilde{\Kzero}(\Lambda)$ of~$\Lambda$ is the kernel of the morphism~$\Kzero(\Lambda)\twoheadrightarrow\Kzero(A)$ in \eqref{equation:bass-tate}.
\end{definition}
In some texts the reduced projective class group is also denoted~$\mathrm{SK}_0$.

Observe that the reduced projective class group is the kernel of a \emph{split} epimorphism, because~$\Kzero(A)\cong\mathbb{Z}$ is projective. So to compute the reduced projective class group it suffices to compute~$\Kzero(\Lambda)$.

The connection between these two types of class groups is given by~\cite[theorem~36.3]{MR0393100} and~\cite[(2.9)]{MR0404410}. The first result says that for a maximal order we have that
\begin{equation}
  \Cl\Lambda\cong\widetilde{\Kzero}(\Lambda),
\end{equation}
whilst the latter describes the ideal class group in general as a \emph{subgroup} of the reduced projective class group via the short exact sequence
\begin{equation}
  \label{equation:SES-Cl-K0red}
  0\to\Cl\Lambda\to\widetilde{\Kzero}(\Lambda)\overset{\lambda_0}{\to}\bigoplus_{\mathfrak{p}\in\Spec R\setminus\{(0)\}}\widetilde{\Kzero}(\Lambda_{\mathfrak{p}})\to 0
\end{equation}
In particular, if~$\Lambda$ is maximal, then~$\lambda_0$ is the zero map: by~\cite[theorem~18.7]{MR0393100} we have indeed that~$\Cl\Lambda_{\mathfrak{p}}=\widetilde{\Kzero}(\Lambda_{\mathfrak{p}})$ is zero.

Moreover, we know by Jacobinski that~$\Cl\Lambda\cong\Cl\Lambda'$, for~$\Lambda\subseteq\Lambda'$ an inclusion of \emph{hereditary} orders~\cite[theorem~40.16]{MR0393100}. In particular it suffices to compute the ideal class group of a maximal order containing~$\Lambda$, provided one starts with an hereditary order.

\begin{remark}
  It is possible to reprove Jacobinski's result using \eqref{equation:SES-Cl-K0red} and the results used in the proof of \cref{proposition:correct-reiten-vandenbergh}: if~$\Lambda$ is an hereditary order, then~$\Kzero(\Lambda_{\mathfrak{p}})\cong\mathbb{Z}^{\oplus r-1}$ for~$\mathfrak{p}$ a maximal ideal of~$R$, where~$r$ is the type of~$\Lambda_{\mathfrak{p}}$, because the last terms of \eqref{equation:bass-tate} reduce to the \emph{split} short exact sequence
  \begin{equation}
    0\to\mathbb{Z}^{\oplus r-1}\to\mathbb{Z}^{\oplus r}\to\mathbb{Z}\to 0.
  \end{equation}
\end{remark}

As an immediate corollary to \cref{propchowexseq} we have the following main result. In particular, by the above discussion we obtain an explicit description of the relative tensor triangular Chow groups in the case of an order~$\Lambda$ over a Dedekind domain~$R$.
\begin{corollary}
  \label{corollary:reduced-projective-class-group-affine}
  We have that
  \begin{equation}
    \CH_0^\Delta(R,\Lambda)\cong\widetilde{\Kzero}(\Lambda).
  \end{equation}
  If~$\Lambda$ is moreover hereditary, then
  \begin{equation}
    \CH_0^\Delta(R,\Lambda)\cong\widetilde{\Kzero}(\Lambda)\cong\Cl\Lambda'\oplus\mathbb{Z}^{r-1}
  \end{equation}
  where~$\Lambda'$ is a maximal order containing~$\Lambda$ and~$r$ is the maximal length of a chain of inclusions of orders.
\end{corollary}
In \cref{subsection:group-rings} we will encounter another situation in which we can express the relative tensor triangular Chow groups in terms of class groups of orders, but there the behaviour with respect to inclusions in maximal orders is different.

\begin{remark}
  \label{remark:one-vs-two-sided}
  In~\cite[theorem~40.9]{MR0393100} a description of the (two-sided) Picard group is given. It combines information about the local type (see \cref{proposition:cDVR-type}) and the ramification. This differs from the tensor triangular Chow groups, for which the local type shows up as copies of~$\mathbb{Z}$, not in the form of torsion quotients.
\end{remark}

\paragraph{Hereditary orders on curves}
Up to now we only looked at hereditary orders on Dedekind domains. In~\cite{MR1825805,MR3597149} the case of hereditary orders on smooth (quasi)projective curves over a field~$k$ is studied, mostly from a representation theory point of view.

Let~$C$ be an irreducible quasiprojective curve over~$\Spec k$. Let~$\mathcal{A}$ be an hereditary order in the central simple~$k(C)$\dash algebra~$A$. In this situation \cref{corollary:reduced-projective-class-group-affine} becomes the following statement.

\begin{corollary}
  \label{corollary:reduced-projective-class-group-projective}
  We have that
  \begin{equation}
    \CH_0^\Delta(C,\mathcal{A})\cong\ker\left( \Kzero(\mathcal{A})\twoheadrightarrow\Kzero(A)\cong\mathbb{Z} \right).
  \end{equation}
\end{corollary}
One can use the results of~\cite{MR1825805} to compute Grothendieck groups of hereditary orders in this setting. The results in this paper are stated only for~$k$ algebraically closed. In this case we have by Tsen's theorem that~$\Br(k)=\Br(k(C))=0$, which means that the central simple~$k(C)$\dash algebra~$A$ is always of the form~$\Mat_n(k(C))$, i.e.~it is unramified.

If~$k$ is not algebraically closed, then one should change the definition of~$r$ in~\cite[proposition~2.1]{MR1825805}: it should only incorporate the local types of the hereditary order, not the ramification of a maximal order containing it. The reason why the definition using ramification works in the algebraically closed case is because every central simple~$k(C)$\dash algebra is automatically unramified, and so is every maximal order. But if~$\Br(k(C))\neq 0$ there are ramified maximal orders.

The correct definition should only account for the length of a chain of orders containing~$\mathcal{A}$ and terminating in a maximal order~$\overline{\mathcal{A}}$. If~$\mathcal{A}$ is itself already maximal we will say that this length is~0.
\begin{proposition}
  \label{proposition:length-of-chain}
  Let~$\mathcal{A}$ be a sheaf of hereditary~$\mathcal{O}_C$\dash orders. Let~$r_p$ be the type of the hereditary~$\mathcal{O}_{C,p}$\dash order~$\mathcal{A}_p$. Then the maximal length of a chain of orders containing~$\mathcal{A}$ is independent of the maximal order in which it terminates and is equal to
  \begin{equation}
    \sum_{p\in C}(r_p-1).
  \end{equation}

  \begin{proof}
    This follows from the proof of~\cite[theorem~40.8]{MR0393100}.
  \end{proof}
\end{proposition}
We can now formulate~\cite[proposition~2.1]{MR1825805} in such a way that it is also valid over non-algebraically closed fields. By the discussion above the formulation of~loc.~cit.~can be misinterpreted if one does not assume throughout that~$k$ is algebraically closed.
\begin{proposition}
  \label{proposition:correct-reiten-vandenbergh}
  Let~$\mathcal{A}$ be a sheaf of hereditary~$\mathcal{O}_C$\dash orders in a central simple~$k(C)$\dash algebra~$A$. Let~$\overline{\mathcal{A}}$ be a maximal order containing~$\mathcal{A}$. Then
  \begin{equation}
    \Kzero(\mathcal{A})\cong\Kzero(\overline{\mathcal{A}})\oplus\mathbb{Z}^{\oplus\rho}
  \end{equation}
  where~$\rho\coloneqq\sum_{p\in C_{(0)}}(r_p-1)$.

  \begin{proof}
    This follows from \cref{proposition:length-of-chain} and~\cite[theorem~1.14]{MR978602}.
  \end{proof}
\end{proposition}
We are now ready to prove the main result for hereditary orders on quasiprojective curves.
\begin{theorem}
  \label{theorem:quasiprojective-curve}
  Let~$\mathcal{A}$ be a sheaf of hereditary~$\mathcal{O}_C$\dash orders in a central simple~$k(C)$\dash algebra~$A$. Let~$\overline{\mathcal{A}}$ be a maximal order containing~$\mathcal{A}$. Then
  \begin{equation}
    \begin{aligned}
      \CH_0^\Delta(C,\mathcal{A})&\cong\Cl(\overline{\mathcal{A}})\oplus\mathbb{Z}^{\oplus\rho} \\
      \CH_1^\Delta(C,\mathcal{A})&\cong\mathbb{Z}
    \end{aligned}
  \end{equation}
  where~$\rho\coloneqq\sum_{p\in C_{(0)}}(r_p-1)$.

  \begin{proof}
    By~\cite[proposition~2.1]{MR1825805} we obtain that
    \begin{equation}
      \Kzero(\mathcal{A})\cong\Kzero(\overline{\mathcal{A}})\oplus\mathbb{Z}^{\oplus\rho}.
    \end{equation}
    Now we apply \cref{corollary:reduced-projective-class-group-projective} to conclude.
  \end{proof}
\end{theorem}

We now discuss some situation in which these Chow groups can be described more explicitly, which reduces to having an explicit description of the ideal class group of a maximal order in this geometric setting.

\begin{corollary}
  \label{corollary:quasiprojective-curve-matrix}
  Let~$k$ be algebraically closed. Then for every~$\mathcal{A}$ as in \cref{theorem:quasiprojective-curve} we have that
  \begin{equation}
    \CH_0^\Delta(C,\mathcal{A})\cong\Pic C\oplus\mathbb{Z}^{\oplus\rho}.
  \end{equation}
  If~$k$ is not algebraically closed the same description holds as long as~$A\cong\Mat_n(k(C))$.

  \begin{proof}
    By Tsen's theorem we know that~$\Br(k(C))=0$, so~$A\cong\Mat_n(k(C))$. The maximal orders in~$A$ are all of the form~$\End_X(\mathcal{E})$ for~$\mathcal{E}$ a vector bundle of rank~$n$, and by Morita theory we can conclude because~$\Kzero(\overline{\mathcal{A}})\cong\Kzero(\mathcal{O}_C)\cong\Pic(C)\oplus\mathbb{Z}$.
  \end{proof}
\end{corollary}

\begin{remark}
  It would be interesting to develop the notion of functoriality for relative tensor triangular Chow groups, as was done for the non-relative case in~\cite{MR3423452}. One example would be the observation that the functor
  \begin{equation}
    -\otimes_R\Mat_n(R)\colon R\mhyphen\mathrm{mod}\to\Mat_n(R)\mhyphen\mathrm{mod}
  \end{equation}
  induces multiplication by~$n$ on the level of Grothendieck groups. In more general settings (e.g.\ inclusions of orders) one expects similar interesting behaviour.
\end{remark}

If~$k$ is not algebraically closed we have an inclusion
\begin{equation}
  \label{equation:brauer-inclusion}
  \Br C\hookrightarrow\Br k(C)
\end{equation}
sending an Azumaya algebra to the central simple algebra at the generic point of~$C$. In the special case of~$C=\mathbb{P}_k^1$ we moreover have that~$\Br(\mathbb{P}_k^1)\cong\Br(k)$.

If the class of the central simple~$k(C)$\dash algebra~$\mathcal{A}_\eta$ in the Brauer group~$\Br(k(C))$ actually comes from~$\Br(C)$ in the inclusion~\eqref{equation:brauer-inclusion} we say that it is \emph{unramified}. Because~$C$ is nonsingular of dimension~1 we have that every maximal order in the unramified central simple algebra~$\mathcal{A}_\eta$ is actually an Azumaya algebra~\cite{MR3461057,MR0121392}, and we can describe the Chow groups up to \emph{controlled} torsion. The situation of \cref{corollary:quasiprojective-curve-matrix} is a special case of this where the Azumaya algebra is split, where~$n=1$.
\begin{corollary}
  \label{corollary:quasiprojective-curve-azumaya}
  Let~$\mathcal{A}$ be an hereditary order as in \cref{theorem:quasiprojective-curve} such that~$\mathcal{A}_\eta$ is an unramified central simple~$k(C)$\dash algebra, and denote~$\rho=\sum(e_i-1)$. Let~$n$ be the degree of~$\mathcal{A}_\eta$ over~$k(C)$. Then
  \begin{equation}
    \CH_0^\Delta(C,\mathcal{A})\otimes_{\mathbb{Z}}\mathbb{Z}[1/n]\cong\left( \Pic C\oplus\mathbb{Z}^{\oplus\rho} \right)\otimes_{\mathbb{Z}}\mathbb{Z}[1/n].
  \end{equation}

  \begin{proof}
    Denote by~$\overline{\mathcal{A}}$ any maximal order containing~$\mathcal{A}$. By the assumptions it is necessarily an Azumaya algebra.

    Using~\cite[corollary~1.2]{MR3056551} we have that there exists an isomorphism
    \begin{equation}
      \label{equation:CH-maximal-up-to-torsion}
      \Kzero(C)\otimes_{\mathbb{Z}}\mathbb{Z}[1/n]\cong\Kzero(\overline{\mathcal{A}})\otimes_{\mathbb{Z}}\mathbb{Z}[1/n],
    \end{equation}
    and by \cref{theorem:quasiprojective-curve} we can conclude.
  \end{proof}
\end{corollary}

\begin{remark}
  In this case op.\ cit.\ gives that the map induced on~$\Kzero$ by~$-\otimes_{\mathcal{O}_C}\mathcal{A}$ has torsion (co)kernel of exponent~$n^4$.
\end{remark}

\paragraph{Maximal orders on surfaces}
There is another invariant in the literature which is a special case of relative Chow groups for orders~\cite[\S3.7]{artin-dejong}. In op.~cit.\ these are defined for a (terminal) maximal order~$\mathcal{A}$ on a (smooth) projective surface~$X$ over an algebraically closed field~$k$.  Here we don't need a precise definition of a terminal maximal order, only that it has finite global dimension~\cite[corollary~3.3.5]{artin-dejong}. Less explicitly so, they have also been defined in a more specific setting in \cite{MR1880659}. In both cases this intersection theory for sheaves of orders is used to show that the center of a quadratic Artin--Schelter regular algebra which is finite over its center is necessarily~$\mathbb{P}^2$.

Comparing definitions, we have that the filtration obtained by the tensor action is the same as the filtration by dimension of support \cref{corverdiervsserre} on the abelian level, which is precisely the filtration used in op.\ cit. They define a divisor group for~$\mathcal{A}$, and as the filtrations are the same we see that
\begin{equation}
  \Div(\mathcal{A})\cong\Cyc_1^\Delta(X,\mathcal{A}).
\end{equation}
Moreover, they define a group~$\mathrm{G}_1(\mathcal{A})$ (not to be confused with higher K-theory of coherent sheaves), using the localization sequence \eqref{equation:chow-as-cokernel}, as the two-dimensional analogue of the reduced projective class group. In particular, combining \eqref{equation:chow-as-cokernel} and~\cite[proposition~3.7.8]{artin-dejong} we have that
\begin{equation}
  \mathrm{G}_1(\mathcal{A})\cong\CH_1^\Delta(X,\mathcal{A}).
\end{equation}
Moreover, in~\cite[proposition~3.7.12]{artin-dejong} an explicit description of~$\mathrm{G}_1(\mathcal{A})$ (and hence the codimension-one Chow group) is given in their situation as
\begin{equation}
  0\to k(X)^\times/\det D^\times\to\CH_1^\Delta(X,\mathcal{A})\to\Pic X\to 0
\end{equation}
where~$D$ is the division algebra over~$k(X)$ Morita equivalent to~$\mathcal{A}_\eta$.

\begin{remark}
  A point not addressed here is the relationship between relative tensor triangular Chow groups for hereditary orders on smooth quasiprojective curves and various Chow groups for ``orbifold curves''. By~\cite{MR2018958} there exists a correspondence between these when working over an algebraically closed base field of characteristic zero. Observe that by~\cite{MR1005008} the Chow groups of the orbifold curve are (up to torsion) the same as the Chow groups of the coarse moduli space. Hence the relative tensor triangular Chow groups of an hereditary order on a smooth quasiprojective curve are different from the Chow groups of its associated orbifold curve, because the stackiness shows up as copies of~$\mathbb{Z}$ and not as torsion.

  This raises at least two questions:
  \begin{enumerate}
    \item is there a purely commutative (relative) setup that recovers the relative Chow groups of the order from the orbifold curve?
    \item is there an analogue of~\cite{MR3423452} identifying the Chow group defined by Vistoli with the tensor triangular Chow group of its derived category?
  \end{enumerate}
\end{remark}

\subsection{Chow groups of (integral) group rings}
\label{subsection:group-rings}
In this section we consider the situation where the scheme~$X$ is~$\Spec R$ for a Dedekind domain~$R$, and the coherent~$\mathcal{O}_X$\dash algebra is given by (the sheafification of) the integral group ring~$RG$, for a finite group~$G$ of order~$n$. Observe that in this situation the global dimension of~$RG$ is often infinite. Especially the case where~$R$ is the ring of integers in an algebraic number field is interesting, where it combines the representation theory of finite groups and algebraic number theory.

As in \cref{subsection:chow-groups} we obtain that we can express in the relative tensor triangular Chow groups in terms of classical invariants, see \cref{theorem:degree-zero-integral-group-ring}.

If we denote~$K$ the field of fractions of~$R$, then we will relax \cref{definition:order} by allowing~$KG$ to be a separable~$K$\dash algebra. By Maschke's theorem this will be the case if the characteristic of~$K$ does not divide~$n$ and~$K$ is a perfect field. We will assume this throughout, and it is of course satisfied in the case where~$K$ is an algebraic number field.

By the Artin--Wedderburn decomposition theorem we have that~$KG$ has a direct product decomposition
\begin{equation}
  KG\cong\prod_{i=1}^t\Mat_{n_i}(D_i)
\end{equation}
whose factors are matrix rings over division rings over~$K$. In particular we allow the conditions in \cref{definition:order} to be relaxed in two directions: we can have multiple factors, and the division algebras can have centers which are larger than~$K$.

This allows us to describe the top degree cycle and Chow groups.

\begin{theorem}
  \label{theorem:top-chow-group-group-ring}
  Let~$R$ be a Dedekind domain such that~$RG$ defines an order in~$KG$. Then
  \begin{equation}
    \Cyc_1^\Delta(R,RG)\cong\CH_1^\Delta(R,RG)\cong\mathbb{Z}^{\oplus t}
  \end{equation}
  where~$t$ is the number of simple factors in the Artin--Wedderburn decomposition of~$KG$.

  \begin{proof}
    This is a straightforward generalisation of \cref{proposition:highest-cycle-group,proposition:highest-chow-group}, taking the more general notion of order into account.
  \end{proof}
\end{theorem}

An easy example of the dependence on the field of fractions is given by considering the group rings~$\mathbb{Z}\Cyc_p$ and~$\mathbb{Z}[\zeta_p]\Cyc_p$, for a cyclic group of prime order~$p\geq 3$, where~$\zeta_p$ is a primitive~$p$th root of unity.
\begin{example}
  \label{example:cyclic-group-rings}
  We have that~$\mathbb{Q}\Cyc_p\cong\mathbb{Q}\times\mathbb{Q}(\zeta_p)$, so
  \begin{equation}
    \label{equation:chow-group-cyclic-group}
    \CH_1^\Delta(\mathbb{Z},\mathbb{Z}\Cyc_p)\cong\mathbb{Z}^{\oplus2}.
  \end{equation}
  On the other hand~$\mathbb{Q}(\zeta_p)\Cyc_p\cong\prod_{i=0}^{p-1}\mathbb{Q}(\zeta_p)$, hence
  \begin{equation}
    \CH_1^\Delta(\mathbb{Z}[\zeta_p],\mathbb{Z}[\zeta_p]\Cyc_p)\cong\mathbb{Z}^{\oplus p}.
  \end{equation}
\end{example}
\begin{remark}
  More generally we have that the integral group ring~$\mathbb{Z}G$ considered as a sheaf of algebras over~$\Spec\mathbb{Z}$ has highest Chow group
  \begin{equation}
    \CH_1^\Delta(\mathbb{Z},\mathbb{Z}G)\cong\mathbb{Z}^t
  \end{equation}
  where~$t$ is the number of conjugacy classes of cyclic subgroups of~$G$~\cite[corollary~13.1.2]{MR0450380}.
\end{remark}

For the zero-dimensional Chow groups we obtain a result similar to \cref{corollary:reduced-projective-class-group-affine}. We will not cover the zero-dimensional cycle groups explicitly: there is no uniform description possible but the techniques of \cref{theorem:top-chow-group-group-ring} go through.
\begin{theorem}
  \label{theorem:degree-zero-integral-group-ring}
  Let~$R$ be a Dedekind domain such that~$RG$ defines an order in~$KG$. Then
  \begin{equation}
    \CH_0^\Delta(R,RG)\cong\widetilde{\Kzero}(RG)\cong\Cl RG.
  \end{equation}

  \begin{proof}
    The first isomorphism follows from \cref{propchowexseq}. The second isomorphism is~\cite[remarks~49.11(iv)]{MR892316}.
  \end{proof}
\end{theorem}
The second isomorphism is indeed somewhat special to the situation of group rings: for an hereditary order~$\Lambda$ we had that~$\Cl\Lambda\cong\Cl\Lambda'$ if~$\Lambda\subseteq\Lambda'$ is an inclusion of orders, reducing the computation of the class group to that of a maximal order. To compute the class group of a group ring, observe that~$RG$ is maximal if and only if it is hereditary, which happens if and only if~$n\in R^\times$~\cite[theorem~41.1]{MR0393100}.

Moreover, the inclusion of~$RG$ into a maximal order~$\Lambda'$ usually only induces an epimorphism of class groups. In particular one obtains a short exact sequence
\begin{equation}
  0\to\derived(RG)\to\Cl(RG)\cong\widetilde{K}_0(RG)\to\Cl(\Lambda')\to 0
\end{equation}
as in~\cite[(49.33)]{MR892316}, independent of the choice of~$\Lambda'$.

In the case where~$R$ is the ring of integers in an algebraic number field, we get by the Jordan--Zassenhaus theorem that~$\Cl RG$ (and therefore~$\CH_0^\Delta(R,RG)$) is a finite abelian group, generalising the theory of class groups and class numbers of~$R$ to the situation of group rings. This is significantly different from the situation for hereditary orders, where the inclusion in a maximal order was responsible for copies~$\mathbb{Z}$ in the Chow groups. More information and some explicit expressions can be found in~\cite{MR0175935,MR0404410}.

To end this discussion we give some examples of explicit computations of~$\Cl\mathbb{Z}G$.

\begin{example}
  If one considers the situation of \cref{example:cyclic-group-rings}, then the (necessarily unique) maximal order in~$\mathbb{Q}\times\mathbb{Q}(\zeta_p)$ is~$\mathbb{Z}\times\mathbb{Z}[\zeta_p]$, and we~\cite[theorem~50.2]{MR892316} we obtain the following
  \begin{equation}
    \CH_0^\Delta(\mathbb{Z},\mathbb{Z}\Cyc_p)\cong\Cl(\mathbb{Z}[\zeta_p]).
  \end{equation}
  The order of this group is the class number of the cyclotomic field~$\mathbb{Q}(\zeta_p)$. For example if~$p=23$ then~$\CH_0^\Delta(\mathbb{Z},\mathbb{Z}\Cyc_{23})\cong\mathbb{Z}/3\mathbb{Z}$.
\end{example}

Using the class numbers of cyclotomic fields it is possible to give a complete classification of the finite abelian groups for which~$\Cl(\mathbb{Z}G)$ (and therefore~$\CH_0^\Delta(\mathbb{Z},\mathbb{Z}G)$) is zero: by~\cite[corollary~50.17]{MR892316} this is only the case if~$G$ is cyclic of order~$\leq 11$, cyclic of order~$13,14,17,19$ or the Klein group of order~4.

\subsection{Chow groups in the singular case}
Finally we discuss a single example where the base is singular, but the order is a noncommutative resolution and in particular has finite global dimension. Observe that this case is covered by the general results in \cref{subsection:main-result}. By no means is this a complete discussion, it is given to suggest possible future research.

We will work in the setting of~\cite[remark~2.7]{MR2854109}. Consider
\begin{equation}
  \begin{aligned}
    R_1\coloneqq k[[x,y]]/(xy),
    R_2\coloneqq k[[x,y]]/(y^2-x^3)
  \end{aligned}
\end{equation}
which are the complete local rings for the nodal (resp.~cuspidal) curve singularity, with maximal ideals~$\mathfrak{m}_i$. Denote their normalizations by~$\widetilde{R}_i$. Then the \emph{Auslander order} is introduced in op.~cit., and it is given by
\begin{equation}
  A_i\coloneqq
  \begin{pmatrix}
    \widetilde{R}_i & \mathfrak{m}_i \\
    \widetilde{R}_i & R_i
  \end{pmatrix}
\end{equation}
It can be seen that these orders have~3 (resp.~2) simple modules, in particular we get the following description of the cycle groups in dimension~0
\begin{equation}
  \begin{aligned}
    \Kzero(A_1\mhyphen\fl)&\cong\mathbb{Z}^{\oplus3}, \\
    \Kzero(A_2\mhyphen\fl)&\cong\mathbb{Z}^{\oplus2}.
  \end{aligned}
\end{equation}

\paragraph{Acknowledgements} We would like to thank Greg Stevenson for interesting discussions, and the referee for remarks which helped to improve the paper. We would like to thank J\o rgen Rennemo for pointing out an inaccuracy in the original statement of \cref{lmacapsubcats}.

Pieter Belmans was supported by a Ph.D. fellowship of the Research Foundation--Flanders (FWO). Sebastian Klein was supported by the ERC grant no.\ 257004--HHNcdMir.

\bibliographystyle{plain}
\bibliography{mr,arxiv,bibliography}

\end{document}